\documentclass[a4paper,12pt,reqno]{amsart}

\usepackage{amsfonts}
\usepackage{amsmath}
\usepackage{amssymb}
\usepackage{cite}
\usepackage{mathrsfs}
\usepackage{enumitem}
\usepackage[colorlinks]{hyperref}

\numberwithin{equation}{section}

\usepackage{graphicx}

\newtheorem{theorem}{Theorem}[section]
\newtheorem{defi}[theorem]{Definition}

\newtheorem{corollary}[theorem]{Corollary}
\newtheorem{prop}[theorem]{Proposition}

\def\Rr{{\mathbb R}}

\def\Zn{{\mathbb Z}^n}
\def\Tn{{\mathbb T}^n}

\def\Rn{{\mathbb R}^n}
\def\R2n{{\mathbb R}^{2n}}

\def\Rr{{\mathbb R}}
\def\Rn{{\mathbb R}^n}

\def\R2{{\mathbb R}^2}
\def\R2n{{\mathbb R}^{2n}}

\def\N0{{\mathbb N}_{0}}

\def\l2h{{\ell^2(\hbar\mathbb Z^n)}}

\begin{document}

\title[Discrete time-dependent wave equations]
{Discrete time-dependent wave equations I. Semiclassical analysis}

\author[Aparajita Dasgupta]{Aparajita Dasgupta}
\address{
	Aparajita Dasgupta:
	\endgraf
	Department of Mathematics
	\endgraf
	Indian Institute of Technology, Delhi, Hauz Khas
	\endgraf
	New Delhi-110016 
	\endgraf
	India
	\endgraf
	{\it E-mail address} {\rm adasgupta@maths.iitd.ac.in}
}
\author[Michael Ruzhansky]{Michael Ruzhansky}
\address{
	Michael Ruzhansky:
	\endgraf
	Department of Mathematics: Analysis, Logic and Discrete Mathematics
	\endgraf
	Ghent University, Belgium
	\endgraf
	and
	\endgraf
	School of Mathematical Sciences
	\endgraf
	Queen Mary University of London
	\endgraf
	United Kingdom
	\endgraf
	{\it E-mail address} {\rm ruzhansky@gmail.com}
}
\author[Abhilash Tushir]{Abhilash Tushir}
\address{
	Abhilash Tushir:
	\endgraf
	Department of Mathematics
	\endgraf
	Indian Institute of Technology, Delhi, Hauz Khas
	\endgraf
	New Delhi-110016 
	\endgraf
	India
	\endgraf
	{\it E-mail address} {\rm abhilash2296@gmail.com}
}

\thanks{ The first author was supported by Core Research Grant, RP03890G,  Science and Engineering
	Research Board,  India. The second
 author was supported by the EPSRC Grant 
EP/R003025, by the FWO Odysseus 1 grant G.0H94.18N: Analysis and Partial Differential Equations and by the  Methusalem programme of the Ghent University Special Research Fund (BOF) (Grantnumber
01M01021).}
\date{\today}


\begin{abstract}
In this paper we consider a semiclassical version of the wave equations with singular H\"older  time-dependent propagation speeds on the lattice $\hbar\mathbb{Z}^{n}$. We allow the propagation speed to vanish leading to the weakly hyperbolic nature of the equations. Curiously, very much contrary to the Euclidean case
considered by Colombini, de Giorgi and Spagnolo \cite{Colombini-deGiordi-Spagnolo-Pisa-1979} and by other authors, the Cauchy problem in this case is well-posed in $\ell^2(\hbar\mathbb{Z}^{n})$. However, we also recover the well-posedness results in the intersection of certain Gevrey and Sobolev spaces in the limit of the semiclassical parameter $\hbar\to 0$.
\end{abstract}
\maketitle

\tableofcontents

\section{Introduction}

In this paper we are interested in the wave equation on the lattice  
$$\hbar\mathbb{Z}^{n}=\{x\in \mathbb{R}^{n}: x=\hbar k, ~ k\in\mathbb{Z}^n\},$$ 
depending on a (small) parameter $\hbar>0$, and in the behaviour of its solutions as $\hbar\to 0$.
The Laplacian on $\hbar\mathbb{Z}^{n}$ is denoted by $\mathcal{L}_{\hbar}$ and is defined by
$$
\mathcal{L}_{\hbar}u(k)=\sum_{j=1}^{n}\left(u(k+\hbar v_j)+u(k-\hbar v_j)\right)-2nu(k),
$$ 
where $v_j$ is the $j^{th}$ basis vector in $\Zn$, having all zeros except for $1$ as the $j^{th}$ component.

The semiclassical analogue of the classical wave equation with time-dependent coefficients on the lattice $\hbar\Zn$ is given by the Cauchy problem
\begin{equation}\label{pde}
\left\{
                \begin{array}{ll}
                 \partial^{2}_{t}u(t,k)-\hbar ^{-2}a(t)\mathcal{L}_{\hbar}u(t,k)+b(t)u(t,k)=f(t, k),~~\textrm{ with }~~t\in(0,T],\\
                  u(0,k)=u_{0}(k),\quad k\in {\hbar\mathbb{Z}^{n}},\\
                 \partial_{t}u(0,k)=u_1(k), \quad k\in {\hbar\mathbb{Z}^{n}},
                \end{array}
              \right.
              \end{equation}
              for the time-dependent propagation speed $a=a(t)\geq 0$ and bounded  real-valued functions $b=b(t)\geq 0$, as well as $f\in L^{\infty}([0,T],\l2h)$.
             The equation \eqref{pde} is the semiclassical discretisation of the classical wave equation on $\mathbb{R}^n$ with the Cauchy problem in the form
             \begin{equation}\label{EQ:wern}
\left\{
                \begin{array}{ll}
                  \partial^{2}_{t}v(t,x)-a(t)\Delta v(t,x)+b(t)v(t,x)=f(t,x), \quad x\in\mathbb{R}^n,\\
                  v(0,x)=u_{0}(x),\\
                   \partial_{t}v(0,x)=u_1(x),
                \end{array}
              \right.
  \end{equation}
  where $\Delta$ is the usual Laplacian on $\Rn$.
Here we can think of the Cauchy data in \eqref{pde} as the evaluations of the Cauchy data from \eqref{EQ:wern} on the lattice $\hbar\Zn$. 

There is an extensive literature concerning \eqref{EQ:wern} going back to the seminal paper by Colombini, de Giorgi and Spagnolo \cite{Colombini-deGiordi-Spagnolo-Pisa-1979}. 
In particular, some very peculiar things can happen with the solvability of \eqref{EQ:wern}.
For example, it has been shown by Colombini and Spagnolo in \cite{Colombini-Spagnolo:Acta-ex-weakly-hyp}  and by Colombini, Jannelli and Spagnolo in \cite{Colombini-Jannelli-Spagnolo:Annals-low-reg} that even in one space dimension, the Cauchy problem \eqref{EQ:wern} does not have to be well-posed in $\mathcal C^\infty$ if $a\in \mathcal C^\infty$ is not strictly positive or if it is in the H\"older class $a\in \mathcal C^\alpha$ for $0<\alpha<1$.
Therefore, to ensure the well-posedness of \eqref{EQ:wern}, one is forced to work in Gevrey spaces $\mathcal \gamma^s$.

In the analysis of \eqref{EQ:wern}, it makes sense to distinguish between the following four cases depending on the properties of the propagation speed $a(t)$, with $\mathcal{C}^{\alpha}([0,T])$ denoting the H\"older class (see \eqref{EQ:Lipconst}):

\begin{itemize}
\item[] Case 1: $a\in {{\rm Lip}}([0,T])$, $\min_{[0,T]} a(t)>0,$
\item[] Case 2: $a\in \mathcal{C}^{\alpha}([0,T]),$ with $0<\alpha<1$, $\min_{[0,T]} a(t)>0,$
\item[] Case 3: $a\in\mathcal{C}^{l}([0,T]),$ with $l\geq 2,$ $a(t)\geq 0,$
\item[] Case 4: $a\in \mathcal{C}^{\alpha}([0,T])$, with $0<\alpha<2$, $a(t)\geq 0.$
\end{itemize} 
Indeed, the proofs of all these cases are different, also yielding results of different type in several cases. Here, Case 1 is the best `classical' situation, Case 2 is lower regularity strictly hyperbolic case, Case 3 is the regular weakly hyperbolic case, and Case 4 is the worst one when $a$ is of low regularity and may also vanish. 
Let us briefly summarise the known results for these cases, which we formulate, for simplicity, for the case $f\equiv 0$.

\begin{itemize}
\item[] Case 1: For any $s\in\Rr$, if the initial Cauchy data $(u_0,u_1)$ are in the Sobolev spaces $H^{s+1}(\Rn)\times H^s(\Rn)$, then there exists a unique solution of \eqref{EQ:wern} in the space $\mathcal C([0,T],H^{s+1}(\Rn))\cap\mathcal C^1([0,T],H^s(\Rn))$.

\item[] Case 2: If the initial Cauchy data $(u_0, u_1)$ are in $ \gamma^{s}\times \gamma^{s}$, then there exists a unique solution of \eqref{EQ:wern} in $\mathcal C^2([0,T],\gamma^{s})$, provided that
\begin{equation}\label{EQ:rn1}
1\leq s <1 +\frac{\alpha}{1-\alpha}.
\end{equation} 

\item[] Case 3: If the initial Cauchy data $(u_0, u_1)$ are in $\gamma^{s}\times \gamma^{s}$, then there exists a unique solution of \eqref{EQ:wern} in $\mathcal C^2([0,T],\gamma^{s})$, provided that
\begin{equation}\label{EQ:rn2}
1\leq s <1 +\frac{l}{2}.
\end{equation} 

\item[] Case 4: If the initial Cauchy data $(u_0, u_1)$ are in $\gamma^{s}\times \gamma^{s}$,  then there exists a unique solution of \eqref{EQ:wern} in $\mathcal C^2([0,T],\gamma^{s})$, provided that
\begin{equation}\label{EQ:rn3}
1\leq s <1 +\frac{\alpha}{2}. 
\end{equation} 
\end{itemize}
The situation in Case 1 is a classical result. Case 2 was shown in \cite{Colombini-deGiordi-Spagnolo-Pisa-1979} with further extensions in \cite{ColKi:02} and \cite{GR:11} for equations of higher order on $\Rr$ and $\Rn$, respectively.  Case 3 was shown, in particular, in \cite{KS}, and the higher dimensional case including low order terms was analysed in \cite{GR:12}. For the analysis in Case 4 on $\Rn$, \cite{GR:11} can be referred.

\medskip
{\em The situation with \eqref{pde} is in a striking difference with the above results for \eqref{EQ:wern}, in the sense that \eqref{pde} is always well-posed in $\l2h$ for all the Cases 1-4. }In some sense, it is natural since no regularity issues on the lattice are involved, so there is no noticeable loss of regularity. 
             
             \begin{theorem}\label{maintheo}
              Let $T>0.$ Assume $b\in L^{\infty}([0,T])$ satisfies $b(t)\geq 0$ for all $t\in[0,T]$. Then in all the cases, Case 1-Case 4, the Cauchy problem \eqref{pde}  is well-posed in $\l2h$. In particular, 
     if $u_0,u_1\in\l2h$ and $f\in L^2([0,T],\l2h)$, then for every $t\in [0,T]$, we have $u(t),\partial_t u(t)\in\l2h$. Moreover,  for each $\hbar>0$, there exists a constant $C_{\hbar,T}>0$ such that 
         \begin{multline}\label{EQ:est0}
\|u(t)\|^{2}_{\l2h}+\|\partial_{t}{u}(t)\|^{2}_{\l2h} 
\\
 \leq  C_{\hbar,T}\left(\|u_0\|^{2}_{\l2h}+\|{u}_1\|^{2}_{\l2h}
 +\|f\|^2_{L^2([0,T];\l2h)}\right),
\end{multline}      
for all $u_0, u_1\in \l2h$ and $f\in L^2([0,T];\l2h)$.
\end{theorem}

At the same time, we can explain why there is such a difference in the well-posedness results between the settings of $\Rn$ and $\hbar\Zn$.
For this we need to specify the constants $C_{\hbar,T}$ appearing in  \eqref{EQ:est0}, especially their dependence on $\hbar$. Clearly, this constant may go to infinity as $\hbar\to 0$, especially in Cases 2-4. The following theorem shows that under the assumptions that the solutions in $\Rn$ exist, they can be recovered in the limit as $\hbar\to 0$. We require a little additional Sobolev regularity to ensure that the convergence results are global on the whole of $\Rn$.

\begin{theorem}\label{limthm}
Let $u$ and $v$ be the solutions of the Cauchy problems \eqref{pde} on $\hbar\Zn$ and \eqref{EQ:wern} on $\Rn$, respectively, with the same Cauchy data $u_0,u_1$. Now, 
\begin{itemize}
	\item[1)] for the Case 1, if 
	$(u_0,u_1)\in H^{s+1}(\Rn)\times H^s(\Rn)$ with $s\geq 5$ for $n\leq3$ and $s>3+\frac{n}{2}$ for $n\geq 4$, and
	\item[2)]  for the Case 2-Case 4, if $(u_0,u_1)\in\left(\gamma^{s}(\mathbb{R}^n)\cap H^{p+1}(\mathbb{R}^n)\right)\times \left(\gamma^{s}(\mathbb{R}^n)\cap H^{p}(\mathbb{R}^n)\right)$ satisfies the assumptions \eqref{EQ:rn1}-\eqref{EQ:rn3} and $p\geq 5$,
\end{itemize}
then for every $t\in [0,T]$, we have
\begin{equation}\label{EQ:conv}
	\|\partial_t u(t)-\partial_t v(t)\|_{\l2h}+\|u(t)-v(t)\|_{\l2h}\to 0\;\textrm{ as }\;\hbar\to 0.
\end{equation}	
\end{theorem}

Thus, when $\hbar\to 0$, for the above mentioned space, we actually recover the well-posedness results on $\Rn$,  where the solution $u$ in the above statement becomes restricted to the lattice $\hbar\Zn$.

We note that the symbolic calculus of pseudo-differential operators on lattices (or pseudo-difference operators) has been developed in \cite{MR4075583}. In particular, the symbol of $\mathcal{L}_{\hbar}$ defined by 
$\sigma_{\mathcal{L}_{\hbar}}(k,\theta)=e^{-2\pi\frac{ i}{\hbar}k\cdot \theta}
\mathcal{L}_{\hbar}(e^{2\pi\frac{ i}{\hbar}k\cdot \theta})$
is given by
\begin{eqnarray}\label{EQ:Lhsymb-i}
\sigma_{\mathcal{L}_{\hbar}}(k,\theta)&=&e^{-2\pi\frac{ i}{\hbar}k\cdot \theta}
\left(\sum\limits_{j=1}^{n}\left(e^{2\pi \frac{i}{\hbar}(k+hv_{j})\cdot\theta}+e^{2\pi \frac{i}{\hbar}(k-hv_{j})\cdot\theta}\right)-2ne^{2\pi\frac{ i}{\hbar}k\cdot \theta}\right)\nonumber\\
&=&\sum_{j=1}^{n}\left(e^{2\pi i \theta_j}+e^{-2\pi i \theta_j}\right)-2n\nonumber\\
&=&2\sum_{j=1}^{n}\cos(2\pi\theta_j)-2n,
\end{eqnarray} 
with $(k,\theta)\in\hbar\mathbb Z^n\times\Tn$, and it is independent of $k$ (and of $\hbar$).
\begin{defi}{(Symbol class $\left.S^{m}\left(\hbar\mathbb{Z}^{n} \times \mathbb{T}^{n}\right)\right)$}. Let $m\in(-\infty,\infty)$.
	We say that a function $\sigma: \hbar\mathbb{Z}^{n} \times \mathbb{T}^{n} \rightarrow \mathbb{C}$   belongs to  $S^{m}\left(\hbar\mathbb{Z}^{n} \times \mathbb{T}^{n}\right)$    if $\sigma(k, \cdot) \in C^{\infty}\left(\mathbb{T}^{n}\right)$ for all $k \in \hbar\mathbb{Z}^{n}$, and for
	all multi-indices $\alpha, \beta$ there exists a positive constant $C_{\alpha, \beta,\hbar}$ such that 
	\begin{equation}
			\left|D_{\theta}^{(\beta)} \Delta_{k}^{\alpha} \sigma(k, \theta)\right| \leq C_{\alpha, \beta,\hbar}(1+|k|)^{m-|\alpha|},
	\end{equation}
	for all $k \in \hbar\mathbb{Z}^{n}$ and $\theta \in \mathbb{T}^{n}$. We denote by  $\mathrm{Op}(\sigma)$ the pseudo-difference operator with symbol $\sigma$ given by
\begin{equation}\label{EQ:Fsymb}
	\mathrm{Op}(\sigma) f(k):=\int_{\mathbb{T}^{n}} e^{2 \pi \frac{i}{\hbar} k \cdot \theta} \sigma(k, \theta) \widehat{f}(\theta) \mathrm{d} \theta, \quad k\in\hbar\Zn,
\end{equation} 
where 
\begin{equation}\label{EQ:ftlat-i}
\widehat{f}(\theta)=\sum_{k\in \hbar\mathbb{Z}^{n}}e^{-2\pi\frac{ i}{\hbar}k\cdot \theta}f(k),~~~\theta\in\mathbb{T}^{n},~k\in \hbar\mathbb{Z}^{n}.
\end{equation} 
\end{defi}
We refer to Section \ref{SEC:proof} for further explanations in this direction, and to \cite{MR4075583} for thorough details. 
The proof of Theorem \ref{maintheo},   given in Section \ref{SEC:proof}, will rely on some elements of the Fourier analysis on the lattice $\Zn$. 
Difference equations on lattices, including Schr\"odinger equations have been studied (e.g. in \cite{RR-AAM,RR-2006,Rab09,Rab2010,Rab13}) by developing the analysis in terms of kernels. The general symbolic calculus of pseudo-differential operators on $\Zn$ has been recently established in \cite{MR4075583}. It allows one to refine other previously known results, e.g. conditions for the boundedness of pseudo-differential operators in the lattice as e.g. in \cite{CR2011}. The symbolic calculus on $\Zn$ can be thought of as a dual one to the calculus developed on the torus in \cite{RT-JFAA,RT-NFAO,RT-Birk} in the framework of more general research on pseudo-differential operators on compact Lie groups \cite{RT-IMRN,Ruzhansky-Turunen:BOOK}.

In Section \ref{SEC:prep},  we establish some results concerning ordinary differential equations that will be instrumental in the proof of Theorem \ref{maintheo}. Moreover, we establish them with explicit expressions on the appearing constants. While this is not needed for the proof of Theorem \ref{maintheo}, this will be useful in the subsequent paper where we will treat the case of distributional $a,b$ and $f$. In Section \ref{SEC:proof}, we prove Theorem \ref{maintheo}. Finally, in Section \ref{SEC:limit}, we discuss the limiting behaviour of solutions to ({\ref{pde}}) in the limit of the semiclassical parameter $\hbar\to 0$.

To simplify the notation, throughout the paper we will be writing $A \lesssim B$ if there exists a constant $C$ independent of the appearing parameters such that $A\leq CB.$

\section{Preparatory estimates}
            \label{SEC:prep}  
              
In the next proposition, we prove certain energy estimates for second order ordinary differential
equations with explicit dependence on parameters. This will be crucial in the
second part of this paper when we will be considering distributional coefficients.
It extends the result obtained in \cite{Taranto-Ruzhansky-WE-graded} from the point of view of inclusion of the lower order terms and the explicit control on the constants.

Notation-wise, we write that $a\in {\mathcal C}^{\alpha}([0,T])$ if for some constant $L_a$, we have
\begin{equation}\label{EQ:Lipconst}
|a(t)-a(s)|\leq L_a |t-s|^\alpha,
\end{equation}       
for all $t,s\in [0,T].$ The smallest $L_a$ in the above inequality will be denoted by 
$\|a\|_{ {\mathcal C}^{\alpha}([0,T])}.$

\begin{prop}\label{odelem}
Let $T>0$ and $\beta>0$ be positive constants, let $b(t)$ be a bounded real-valued function and let $a(t)$ be a function satisfying one of the following conditions: 
\begin{itemize}
\item[] Case 1: $a\in {{\rm Lip}}([0,T])$, $a_0:=\min_{[0,T]} a(t)>0,$
\item[] Case 2: $a\in \mathcal{C}^{\alpha}([0,T]),$ with $0<\alpha<1$, $a_0:=\min_{[0,T]} a(t)>0,$
\item[] Case 3: $a\in\mathcal{C}^{l}([0,T]),$ with $l\geq 2,$ $a(t)\geq 0,$
\item[] Case 4: $a\in \mathcal{C}^{\alpha}([0,T])$, with $0<\alpha<2$, $a(t)\geq 0.$
\end{itemize} 
Consider the following Cauchy problem:
\begin{equation}\label{ode}
\left\{
                \begin{array}{ll}
                  v^{\prime\prime}(t)+\beta^{2}a(t)v(t)+b(t)v(t)=0, \quad t\in(0,T],\\
                  v(0)=v_{0}\in\mathbb{C},\\
                 v^{\prime}(0)=v_1\in\mathbb{C}.
                \end{array}
              \right.
  \end{equation}
  
  Then the following holds:
  \begin{itemize}
  	\item[Case 1 :] There exist two positive constants $C_{1},K_{1}>0$  such that  for all $t\in [0,T]$, we have 
  	\begin{equation}\label{EQ:ODE1}
  	|v(t)|^{2}+|v^{\prime}(t)|^{2}\leq 
  	C_{1}e^{K_{1}T\langle\beta\rangle^{2}}(|v_0|^2+|v_1|^2),
  	  	\end{equation} 
    	holds for all $\beta>0$.
  	\item[Case 2 :] There exist two positive constants $C_{2},K_{2}>0$  such that for all $t\in [0,T]$, we have 
  	\begin{equation}\label{odecase2}
  	|v(t)|^2+|v^{\prime}(t)|^2\leq  C_{2}e^{K_{2}T\langle\beta\rangle^{\frac{1}{s}}}( |v_0|^2+|v_1|^2),
  	\end{equation}
for any $0<s\leq\frac{1}{2}$ and for all $\beta>0$.
  	\item[Case 3 :] There exist two positive constants $C_{3},K_{3}>0$  such that for all $t\in [0,T]$, we have
  	\begin{equation}\label{odecase3}  
  	|v(t)|^{2}+|v'(t)|^{2}\leq C_{3}e^{K_{3}\langle\beta\rangle^{6-\frac{4}{\sigma}}}(|v_{0}|^{2}+|v_{1}|^{2}),
  	\end{equation} 
with $\sigma=1+\frac{l}{2}$ and  for all $\beta>0$.
  	\item[Case 4 :] There exist two positive constants $C_{4},K_{4}>0$ such that for all $t\in [0,T]$, we have
  	\begin{equation}\label{odecase4}  
|v(t)|^2+|v^{\prime}(t)|^2\leq  C_{4}e^{K_{4}T\langle\beta\rangle^{\frac{1}{s}}}( |v_0|^2+|v_1|^2),
  	\end{equation} 
for any $0<s\leq \frac{\alpha+2}{4}$ and for all $\beta>0$.
  \end{itemize}
The constants $C_{j}$'s and $K_{j}$'s for $j=1,2,3~~ \text{and}~~4$ may depend on $T$ but not on $\beta$.
  \end{prop}
  
  \begin{proof}
  First, we reduce the problem \eqref{ode} to a first order equation. We define
  \begin{equation}
V(t) := \left(
    \begin{matrix}
        iv(t) \\
        \partial_{t}v(t)
    \end{matrix}\right),
  ~~~~~   V_{0} := \left(
   \begin{matrix}
   i {v_0}\\
  { v_1}\end{matrix}\right),
    \end{equation}
and the matrices 
\begin{equation}
A(t):=\left(\begin{matrix}
0&1\\
a(t)&0
\end{matrix}\right),~~~
B(t):= \left(\begin{matrix}
0&1-\beta^2\\
b(t)&0
\end{matrix}\right).
\end{equation}
This allows us to reformulate the given second order system \eqref{ode} as the first order system
\begin{equation}\label{Estode}
\left\{
\begin{array}{ll}
V^{\prime}=i\beta^{2} A(t)V(t)+iB(t)V(t),\\
  V(0)=V_{0}.
\end{array}\right.
\end{equation}
Now we will treat each case separately.

\medskip 
\textbf{Case 1: $a\in \text{Lip}([0,T]), a_0:=\min_{[0,T]} a(t)>0$}. 

\smallskip
We observe that the eigenvalues of the matrix $A(t)$ are given by $\pm\sqrt{a(t)}.$ The symmetriser $S$ of $A,$ i.e., the matrix such that 
$$SA-A^{\ast}S=0,
$$ is given by
$$S(t)=\left(\begin{matrix}
a(t)&0\\
0&1
\end{matrix}\right).$$
Consider
\begin{eqnarray}\label{eqq1}
\left(S(t)V(t),V(t)\right)&=&a(t)|v(t)|^{2}+|v^{\prime}(t)|^{2}\nonumber\\
&\leq& \max_{t\in[0,T]}\{a(t),1\}|v(t)|^{2}+\max_{t\in[0,T]}\{a(t),1\}|v^{\prime}(t)|^{2}\nonumber\\
&=&\max_{t\in[0,T]}\{a(t),1\}|V(t)|^2.
\end{eqnarray}
Similarily
\begin{eqnarray}\label{eqq2}
\left(S(t)V(t),V(t)\right)&\geq&\min_{t\in[0,T]}\{a(t),1\}|V(t)|^2.
\end{eqnarray}
If we now define the energy
$$E(t):=\left(S(t)V(t),V(t)\right),$$ 
then from \eqref{eqq1}, and \eqref{eqq2}, it follows that
\begin{equation}\label{eqq3}
\min_{t\in[0,T]}\{a(t),1\}|V(t)|^{2}\leq E(t)\leq \max_{t\in[0,T]}\{a(t),1\}|V(t)|^{2}.	
\end{equation}
Since $a\in {\textbf{Lip}}([0,T])$, there exist two constants $a_0$ and $a_1$ such that
$$a_0=\min\limits_{t\in[0,T]}a(t) ~~~\textrm{ and }~~~a_1=\max\limits_{t\in[0,T]}a(t).$$
Thus, if we set
$ c_{0,a}= \min\{a_0,1\}$ and $c_{1,a}=\max\{a_1,1\}$, then
the inequality \eqref{eqq3} becomes
\begin{equation}\label{eqq4}
c_{0,a}|V(t)|^{2}\leq E(t)\leq c_{1,a}|V(t)|^{2}.	
\end{equation}
Then we can calculate
\begin{eqnarray}\label{eninq3}
E_{t}(t)&=&\left(S_{t}(t)V(t), V(t)\right)+\left(S(t)V_{t}(t), V(t)\right)+\left(S(t)V(t), V_t(t)\right)\nonumber\\
&=& \left(S_{t}(t)V(t), V(t)\right)+i\beta^{2}\left(S(t)A(t)V(t), V(t)\right)+i\left(S(t)B(t)V(t), V(t)\right)-\nonumber\\
&&i\beta^{2}\left(S(t)V(t), A(t)V(t)\right)-i\left(S(t)V(t), B(t)V(t)\right)\nonumber\\
&=& \left(S_{t}(t)V(t), V(t)\right)+i\beta^{2}\left(\left(SA-A^{\ast}S\right)V(t), V(t)\right)+\nonumber\\
&&i\left(\left(SB-B^{\ast}S\right)V(t), V(t)\right)\nonumber\\
&=&\left(S_{t}(t)V(t), V(t)\right)+i\left(\left(SB-B^{\ast}S\right)V(t), V(t)\right). 
\end{eqnarray}
From the definition of $S$ and $B$, we get
\begin{eqnarray*}\label{bdd}
SB-B^{\ast}S 
&=&\left(\begin{matrix}
0&a(1-\beta^2)-b\\
b-a(1-\beta^2)&0
\end{matrix}\right).\nonumber
\end{eqnarray*}
Then
\begin{eqnarray}\label{eqq5}
	\|SB-B^{\ast}S\|_{L^{\infty}([0,T])}&\leq&\langle\beta\rangle^{2}\|a\|_{L^{\infty}([0,T])}+\|b\|_{L^{\infty}([0,T])}.
\end{eqnarray}
From using \eqref{eqq5} in \eqref{eninq3}, we get
\begin{eqnarray}\label{Ineqen3}
E_{t}(t)&\leq& \|S_t(t)\|_{L^{\infty}([0,T])} |V(t)|^2+(\langle\beta\rangle^{2}\|a\|_{L^{\infty}([0,T])}+\|b\|_{L^{\infty}([0,T])}) |V(t)|^{2}\nonumber\\
&=&\left(\|a^{\prime}\|_{L^{\infty}([0,T])} +\langle\beta\rangle^{2}\|a\|_{L^{\infty}([0,T])}+\|b\|_{L^{\infty}([0,T])}\right)|V(t)|^{2}.
\end{eqnarray}
Applying  \eqref{eqq4} in \eqref{Ineqen3}, we get
\begin{eqnarray}\label{Inenend}
E_{t}(t)&\leq&c_{0,a}^{-1}\left(\|a^{\prime}\|_{L^{\infty}([0,T])} +\langle\beta\rangle^{2}\|a\|_{L^{\infty}([0,T])}+\|b\|_{L^{\infty}([0,T])}\right)E(t)\nonumber\\
&=&\kappa E(t),
\end{eqnarray}
where $\kappa=c_{0,a}^{-1}\left(\|a^{\prime}\|_{L^{\infty}([0,T])} +\langle\beta\rangle^{2}\|a\|_{L^{\infty}([0,T])}+\|b\|_{L^{\infty}([0,T])}\right)>0$. Applying the Gronwall lemma to \eqref{Inenend}, we deduce that 
\begin{eqnarray}\label{energy}
E(t)
\leq e^{\kappa T}E(0)=e^{c_{0,a}^{-1}\left(\|a^{\prime}\|_{L^{\infty}([0,T])} +\langle\beta\rangle^{2}\|a\|_{L^{\infty}([0,T])}+\|b\|_{L^{\infty}([0,T])}\right)T}E(0).
\end{eqnarray}
Therefore, putting together \eqref{energy} and \eqref{eqq4} we obtain
\begin{eqnarray}\label{c1:eqn1}
c_{0,a}|V(t)|^{2}\leq E(t)&\leq& 
e^{c_{0,a}^{-1}\left(\|a^{\prime}\|_{L^{\infty}([0,T])} +\langle\beta\rangle^{2}\|a\|_{L^{\infty}([0,T])}+\|b\|_{L^{\infty}([0,T])}\right)T}E(0)\nonumber\\
&\leq& c_{1,a}e^{c_{0,a}^{-1}\left(\|a^{\prime}\|_{L^{\infty}([0,T])} +\langle\beta\rangle^{2}\|a\|_{L^{\infty}([0,T])}+\|b\|_{L^{\infty}([0,T])}\right)T}|V(0)|^{2}.\nonumber\\
\end{eqnarray} 
If we set,  $C_{1}=c_{0,a}^{-1}c_{1,a}e^{c_{0,a}^{-1}\left(\|a^{\prime}\|_{L^{\infty}([0,T])}+\|b\|_{L^{\infty}([0,T])}\right)T}$, and  $K_{1}=c_{0,a}^{-1}\|a\|_{L^{\infty}([0,T])}$, then using the definition of $V(t)$ and \eqref{c1:eqn1}, we obtain \eqref{EQ:ODE1}, that is,
$$
|v(t)|^{2}+|v^{\prime}(t)|^{2}\leq 
C_{1}e^{K_{1}T\langle\beta\rangle^{2}}(|v_0|^2+|v_1|^2),
$$ 
finishing the proof of Case 1.

\medskip
\textbf{Case 2:} ~ $a\in \mathcal{C}^{\alpha}([0,T]),$ with $0<\alpha<1$, $a_0:=\min_{[0,T]} a(t)>0.$

\smallskip
Here we follow the
method developed by Colombini and Kinoshita in \cite{ColKi:02}. We look for solutions of the form 
\begin{equation}\label{EQ:vw}
V(t)=e^{-\rho(t)\langle\beta\rangle^{\frac{1}{s}}}(\det H(t))^{-1}H(t)W(t),
\end{equation} 
where 
\begin{itemize}
\item $s\in\mathbb{R}$ depends on $\alpha$ as will be determined in the argument;
\item the function $\rho\in\mathcal{C}^{1}([0,T])$ is real-valued  with $\rho(0)=0$;
\item $W(t)$ is the energy;
\item $H(t)$ is the matrix defined by
$$H(t)=\left( \begin{matrix}
1&1\\
\lambda_{1}^{\epsilon}(t)&\lambda_{2}^{\epsilon}(t)
\end{matrix}   \right),$$
where for all $\epsilon\in(0,1],$ $\lambda_{1}^{\epsilon}(t)$ and $\lambda_{2}^{\epsilon}(t)$ are regularisations of the eigenvalues of the matrix $A(t)$ of the form
$$\lambda_{1}^{\epsilon}(t):=\left(-\sqrt{a}\ast\phi_{\epsilon}\right)(t),$$
$$\lambda_{2}^{\epsilon}(t):=\left(+\sqrt{a}\ast\phi_{\epsilon}\right)(t),$$
with $\{\phi_{\epsilon}(t)\}_{\epsilon>0}$, being a family of cut-off functions defined starting from a non-negative even function $\phi\in\mathcal{C}_{0}^{\infty}(\mathbb{R}),$ with $\int\limits_{\mathbb{R}}\phi=1,$ by setting $\phi_{\epsilon}(t):= \frac{1}{\epsilon}\phi(\frac{1}{\epsilon}).$ By construction, it follows that $\lambda_{1}^{\epsilon},\lambda_{2}^{\epsilon}\in \mathcal{C}^{\infty}([0,T]).$ 
\end{itemize}
We observe the inequality
\begin{eqnarray}\label{determinantes}
\det H(t)=\lambda_{2}^{\epsilon}(t)-\lambda_{1}^{\epsilon}(t) \geq 2\sqrt{a_0}.
\end{eqnarray} 
Moreover, using the H\"{o}lder regularity of $a(t)$ of order $\alpha$ and,  for all $t\in[0,T]$, we have  
\begin{eqnarray}
\left|\lambda_{1}^{\epsilon}(t)+\sqrt{a(t)}\right|&=&\left|\left(-\sqrt{a}\ast\phi_{\epsilon}\right)(t)+\sqrt{a(t)}\right|\nonumber\\
&=&\left|\int_{\mathbb{R}}\sqrt{a(t-y)}\phi_{\epsilon}(y)dy-\sqrt{a(t)}\int_{\mathbb{R}}\phi(y)dy\right|\nonumber\\
&=&\left|\int_{\mathbb{R}}\sqrt{a(t-\epsilon x)}\phi(x)dx-\sqrt{a(t)}\int_{\mathbb{R}}\phi(x)dx\right|\nonumber\\
&\leq&\int_{\mathbb{R}}\frac{|a(t-\epsilon x)-a(t))|}{\sqrt{a(t-\epsilon x)}+\sqrt{a(t)}}\phi(x)dx\nonumber\\
&\leq&\frac{\|a\|_{ {\mathcal C}^{\alpha}([0,T])}}{2\sqrt{a_{0}}} \epsilon^{\alpha}.\nonumber
\end{eqnarray}
Similarly, we can compute $|\lambda_{2}^{\epsilon}(t)-\sqrt{a(t)}|$ and we get 
\begin{eqnarray}\label{lambes1}
\left|\lambda_{1}^{\epsilon}(t)+\sqrt{a(t)}\right|&\leq& c(a)\epsilon^{{\alpha}},\nonumber\\
\left|\lambda_{2}^{\epsilon}(t)-\sqrt{a(t)}\right|&\leq& c(a)\epsilon^{{\alpha}},
\end{eqnarray}
where 
\begin{eqnarray}\label{c(a)}
	c(a)=\frac{\|a\|_{ {\mathcal C}^{\alpha}([0,T])}}{2\sqrt{a_0}}.
\end{eqnarray}
Now substituting the suggested solution \eqref{EQ:vw} in \eqref{Estode}, we get
\begin{multline*}
-\rho^{\prime}(t)\langle\beta\rangle^{\frac{1}{s}}e^{-\rho(t)\langle\beta\rangle^{\frac{1}{s}}}\frac{H(t)W(t)}{\det H(t)}+e^{-\rho(t)\langle\beta\rangle^{\frac{1}{s}}}\frac{H_t(t)W(t)}{\det H(t)}+e^{-\rho(t)\langle\beta\rangle^{\frac{1}{s}}}\frac{H(t)W_t(t)}{\det H(t)} \\
-e^{-\rho(t)\langle\beta\rangle^{\frac{1}{s}}}(\det H)_{t}(t)\frac{H(t)W(t)}{(\det H(t))^{2}}
\\ =i\beta^{2} A(t)e^{-\rho(t)\langle\beta\rangle^{\frac{1}{s}}}\frac{H(t)W(t)}{\det H(t)}+iB(t)e^{-\rho(t)\langle\beta\rangle^{\frac{1}{s}}}\frac{H(t)W(t)}{\det H(t)}. 
\end{multline*} 
Multiplying both sides by $e^{\rho(t)\langle\beta\rangle^{\frac{1}{s}}}\det H(t) H^{-1}(t)$, we get
\begin{multline*}
W_t(t)=\rho^{\prime}(t)\langle\beta\rangle^{\frac{1}{s}}W(t)-H^{-1}(t)H_{t}(t)W(t)+(\det H)_{t}(t)(\det H(t) )^{-1}W(t)
\\
+ i\beta^{2} H^{-1}(t)A(t)H(t)W(t)+i H^{-1}(t)B(t)H(t)W(t).\nonumber
\end{multline*}
This leads to the equality
\begin{eqnarray}
& &\frac{d}{dt}\left|W(t)\right|^2 \nonumber\\
&=& 2{\rm Re}(W_t(t),W(t))\nonumber\\
&=& 2\rho^{\prime}(t)\langle\beta\rangle^{\frac{1}{s}}|W(t)|^{2}-2{\rm Re}\left(H^{-1}(t)H_{t}(t)W(t), W(t)\right)+\nonumber\\
&& 2(\det H(t) )^{-1}(\det H)_{t}(t)|W(t)|^{2}+2\beta^{2} {\rm Im}\left(H^{-1}(t)A(t)H(t)W(t),W(t)\right)+\nonumber\\
&& 2 {\rm Im}\left(H^{-1}(t)B(t)H(t)W(t),W(t)\right).\nonumber
\end{eqnarray}
We note that 
\begin{equation}
2{\rm Im}\left(H^{-1}AHW,W\right)\leq \| H^{-1}AH-(H^{-1}AH)^{\ast}\| \|W\|^{2}.
\end{equation}
Thus, we obtain
\begin{multline}{\label{estimate}}
\frac{d}{dt}\left|W(t)\right|^2\leq  \big(2\rho^{\prime}(t)\langle\beta\rangle^{\frac{1}{s}}+2 \| H^{-1}(t)H_{t}(t) \|+2|\left(\det H(t)\right)^{-1}\left(\det H\right)_{t}(t)|+\big.\\
\big.\langle\beta\rangle^{2} \| H^{-1}AH-(H^{-1}AH)^{\ast}\|+\| H^{-1}BH-(H^{-1}BH)^{\ast}\|\big)\|W\|^{2}.
\end{multline}
Using the techniques from \cite{GR-sublaplacian,Taranto-Ruzhansky-WE-graded}, we estimate the above terms as follows:
\begin{enumerate}
\item $\|H^{-1}(t)H_{t}(t)\|\lesssim { \frac{\|a\|_{{\mathcal C}^{\alpha}([0,T])}}{a_0}}\epsilon^{{\alpha}-1}$, 
\item $|\left(\det H(t)\right)^{-1}\left(\det H\right)_{t}(t)|\lesssim
\frac{\|a\|_{{\mathcal C}^{\alpha}([0,T])}}{a_0} \epsilon^{{\alpha}-1}$, 
\item $\|H^{-1}AH-(H^{-1}AH)^{\ast}\|\lesssim
\frac{\|a\|_{{\mathcal C}^{\alpha}([0,T])}\|\sqrt{a}\|_{L^{\infty}([0,T])}}{a_0}
\epsilon^{{\alpha}}$,
\item $\|H^{-1}BH-(H^{-1}BH)^{\ast}\|\lesssim \langle\beta\rangle^{2}\frac{\|b\|_{L^\infty([0,T])}+\|\sqrt{a}\|^{2}_{L^{\infty}([0,T])}}{\sqrt{a_0}}.$
\end{enumerate}
Indeed, we deal with these four terms as follows: 
\begin{enumerate}
\item Since $H^{-1}(t)=\frac{1}{\lambda^{\epsilon}_2-\lambda^{\epsilon}_1}\left(\begin{matrix}
\lambda^{\epsilon}_2 &-1\\
-\lambda^{\epsilon}_1& 1
\end{matrix}\right)$ and $H_{t}(t)=\left(\begin{matrix}
0 & 0\\
\partial_{t}\lambda^{\epsilon}_1 & \partial_{t}\lambda^{\epsilon}_2
\end{matrix}\right)$,  it follows that the entries of the matrix $H^{-1}H_{t}$ are given by the functions $\frac{\partial_{t}\lambda^{\epsilon}_{j}}{\lambda^{\epsilon}_2-\lambda^{\epsilon}_1}.$ We compute for $j=2$,  i.e.,  $\partial_{t}\lambda^{\epsilon}_{2}(t)$, can be estimated by
\begin{eqnarray}\label{pares}
\left|\partial_{t}\lambda^{\epsilon}_{2}(t)\right|&= &\left|\sqrt{a}\ast\partial_{t}\phi_{\epsilon}(t)\right|=\left|\frac{1}{\epsilon}\int_{\mathbb{R}}\sqrt{a(t-\rho\epsilon)}\phi^{\prime}(\rho)d\rho\right|\nonumber\\
&\leq&\frac{1}{\epsilon}\int_{\mathbb{R}}\left|\sqrt{a(t-\rho\epsilon)}-\sqrt{a(t)}\right|\phi^{\prime}(\rho)d\rho\nonumber\\
&=&\frac{1}{\epsilon}\int_{\mathbb{R}}\frac{|a(t-\rho\epsilon )-a(t)|}{\sqrt{a(t-\rho\epsilon )}+\sqrt{a(t)}}\phi^{\prime}(\rho)d\nonumber\rho\\
&\leq&c(a)\epsilon^{\alpha-1},
\end{eqnarray}
where $c(a)$ given by \eqref{c(a)} comes from the H\"older continuity of $a(t)$, and also using $\int\limits_{\mathbb{R}} \phi^{\prime}=0.$
Then using \eqref{determinantes}, and \eqref{pares}, we get
$$\|H^{-1}(t)H_{t}(t)\|\leq  \frac{\|a\|_{{\mathcal C}^{\alpha}([0,T])}}{4a_0}\epsilon^{\alpha-1}\lesssim  \frac{\|a\|_{{\mathcal C}^{\alpha}([0,T])}}{a_0}\epsilon^{\alpha-1}.$$

\item We can estimate 
$$\left|(\det H(t))^{-1}(\det H)_{t}(t)\right|=\left|\frac{\partial_{t}\lambda^{\epsilon}_2-\partial_{t}\lambda^{\epsilon}_1}{\lambda^{\epsilon}_2-\lambda^{\epsilon}_1}\right|\leq 
\frac{\|a\|_{{\mathcal C}^{\alpha}([0,T])}}{2a_0}\epsilon^{{\alpha}-1}\lesssim \frac{\|a\|_{{\mathcal C}^{\alpha}([0,T])}}{a_0}\epsilon^{{\alpha}-1}.$$

\item In this case, we calculate the matrix that we are interested in, that is,
\begin{equation}\label{EQ:Hs}
H^{-1}AH-(H^{-1}AH)^{\ast}=\frac{1}{\lambda^{\epsilon}_2-\lambda^{\epsilon}_1}\left(\begin{matrix}
0 & -2a+(\lambda^{\epsilon}_1)^{2}+(\lambda^{\epsilon}_2)^{2}\\
2a-(\lambda^{\epsilon}_1)^{2}-(\lambda^{\epsilon}_2)^{2} & 0
\end{matrix}\right).
\end{equation} 
Since $(\lambda^{\epsilon}_1)^{2}=(\lambda^{\epsilon}_2)^{2}$, and recalling inequality \eqref{determinantes}, we see that the desired norm can be estimated if we estimate the function $|a(t)-(\lambda^{\epsilon}_2)^{2}|$ only.
 A straightforward calculation using \eqref{lambes1}, shows that 
 \begin{eqnarray}\label{c2:eq2}
 \left|a(t)-(\lambda^{\epsilon}_2)^{2}\right|&=& \left|\left(\sqrt{a(t)}-(\lambda^{\epsilon}_2)\right)\left(\sqrt{a(t)}+(\lambda^{\epsilon}_2)\right)\right|\nonumber\\
 &\leq& c(a)\epsilon^{{\alpha}} \left|\sqrt{a(t)}+\int_{\mathbb{R}}\sqrt{a(t-s)}\phi_{\epsilon}(s)ds\right|\nonumber\\
 &\leq& 2c(a)\|\sqrt{a}\|_{L^{\infty}([0,T])}\epsilon^{{\alpha}},  
 \end{eqnarray}
 where $c(a)$ given by \eqref{c(a)} comes from the H\"older continuity of $a(t)$.
Using \eqref{c2:eq2}, and \eqref{determinantes}, it follows that
\begin{eqnarray}
	\|H^{-1}AH-(H^{-1}AH)^{\ast}\|&\leq& \frac{2c(a)\|\sqrt{a}\|_{L^{\infty}([0,T])} }{\sqrt{a_{0}}}\epsilon^{{\alpha}}\nonumber\\
	&\lesssim& 
	\frac{\|a\|_{{\mathcal C}^{\alpha}([0,T])}\|\sqrt{a}\|_{L^{\infty}([0,T])}}{a_{0}} \epsilon^{{\alpha}}.
\end{eqnarray}
 
 \item
Next we estimate the term $\|H^{-1}BH-(H^{-1}BH)^{\ast}\|.$
For that, we explicitly write the matrix we are interested in, that is

\begin{equation}\label{c2:m1}
H^{-1}BH-(H^{-1}BH)^{\ast}
= \frac{1}{\lambda^{\epsilon}_2-\lambda^{\epsilon}_1}
\left(\begin{matrix}
0&2(1-\beta^{2})(\lambda_{1}^{\epsilon})^{2}-2b\\
-2(1-\beta^{2})(\lambda_{1}^{\epsilon})^{2}+2b&0
\end{matrix}\right).
\end{equation} 
Consequently, using \eqref{determinantes}, we get
\begin{eqnarray}\label{EQ:estbs}
\|H^{-1}BH-(H^{-1}BH)^{\ast}\|&\leq&2\left(\frac{\|b\|_{L^\infty([0,T])}+\langle\beta\rangle^{2}\|\sqrt{a}\|^{2}_{L^{\infty}([0,T])}}{\sqrt{a_0}}
\right)\nonumber\\
&\lesssim& \langle\beta\rangle^{2}\frac{\|b\|_{L^\infty([0,T])}+\|\sqrt{a}\|^{2}_{L^{\infty}([0,T])}}{\sqrt{a_0}}.
\end{eqnarray} 
\end{enumerate}

Combining estimates (1)-(4) above, we get the estimate for the derivative of the energy, that is
\begin{multline}\label{est.2}
\frac{d}{dt}|W(t)|^2\lesssim \left(2\rho^{\prime}(t)\langle\beta\rangle^{\frac{1}{s}}+
 { \frac{\|a\|_{{\mathcal C}^{\alpha}([0,T])}}{a_0}}\epsilon^{{\alpha}-1} + 
\langle\beta\rangle^{2}\frac{\|a\|_{{\mathcal C}^{\alpha}([0,T])}\|\sqrt{a}\|_{L^{\infty}([0,T])}}{a_0}
\epsilon^{{\alpha}}
\right. \\
+ \left.
 \langle\beta\rangle^{2}\frac{\|b\|_{L^\infty([0,T])}+\|\sqrt{a}\|^{2}_{L^{\infty}([0,T])}}{\sqrt{a_0}}
\right)|W(t)|^2 . 
\end{multline}
We choose $\epsilon=\langle\beta\rangle^{-2}$ and define $\rho(t):=\rho(0)-K_{2}t=-K_{2}t$, for some $K_{2}>0$ to be specified. Substituting it in \eqref{est.2}, we get 
\begin{multline}\label{est3}
 \frac{d}{dt}|W(t)|^2 
\lesssim\left(-2K_{2}\langle\beta\rangle^{\frac{1}{s}}
+  \frac{\|a\|_{{\mathcal C}^{\alpha}([0,T])}}{a_0}\langle\beta\rangle^{2(1-{\alpha})} 
+\frac{\|a\|_{{\mathcal C}^{\alpha}([0,T])}\|\sqrt{a}\|_{L^{\infty}([0,T])}}{a_0}\langle\beta\rangle^{2(1-{\alpha})}
\right. \\
 \quad + \left. 
\langle\beta\rangle^{2}\frac{\|b\|_{L^\infty([0,T])}+\|\sqrt{a}\|^{2}_{L^{\infty}([0,T])}}{\sqrt{a_0}}
\right)|W(t)|^2.
\end{multline}
Since $0<\alpha<1$, this implies
\begin{equation}
 \frac{d}{dt}|W(t)|^2
\lesssim
\left(-2K_{2}\langle\beta\rangle^{\frac{1}{s}}
+  \kappa\langle\beta\rangle^{2}
\right)|W(t)|^2,	
\end{equation}
where
\begin{equation}
	\kappa=\frac{\|a\|_{{\mathcal C}^{\alpha}([0,T])}}{a_0} 
	+\frac{\|a\|_{{\mathcal C}^{\alpha}([0,T])}\|\sqrt{a}\|_{L^{\infty}([0,T])}}{a_0}
	+ 
	\frac{\|b\|_{L^\infty([0,T])}+\|\sqrt{a}\|^{2}_{L^{\infty}([0,T])}}{\sqrt{a_0}}.
\end{equation}
  If we choose $K_{2}=\frac{\kappa}{2}$ and
	$\dfrac{1}{s}\geq 2$, then
for all $t\in[0,T]$ and $\beta>0$, we have 
\begin{equation}\label{EQ:mon-W}
\frac{d}{dt}|W(t)|^2\leq 0.
\end{equation} 
Now by using monotonicity of $|W(t)|$ and $\rho(0)=0$, we get
\begin{equation}\label{case2:w1}
|W(t)|\leq \|H(0)\|^{-1}|\det H(0)||V(0)|.
\end{equation}
Now from \eqref{EQ:vw} and \eqref{case2:w1}, we get
\begin{eqnarray}\label{case2:w2}
|V(t)|&\leq&\|H(t)\|\left|\det(H(t))\right|^{-1}e^{K_{2}t\langle\beta\rangle^{\frac{1}{s}}}\left|W(t)\right|\nonumber\\
&\leq&{\|H(t)\|}{\|H(0)\|^{-1}\left|\det H(t)\right|^{-1}}\left|\det H(0)\right|e^{K_{2}t\langle\beta\rangle^{\frac{1}{s}}}|V(0)|.
\end{eqnarray}
Note that the definition of $\lambda_{1}^{\epsilon}, \lambda_{2}^{\epsilon}$ and 
\eqref{determinantes}
imply
\begin{equation}\label{c2:eq1}
{\|H(t)\|}{\|H(0)\|^{-1}\left|\det H(t)\right|^{-1}}\left|\det H(0)\right|\lesssim 
\frac{\|\sqrt{a}\|^2_{L^\infty([0,T])}}{\sqrt{a_0}}=
\frac{\|{a}\|_{L^\infty([0,T])}}{\sqrt{a_0}}.
\end{equation}
Then \eqref{case2:w2} and \eqref{c2:eq1} allow one to estimate
\begin{eqnarray}\label{c2:eqn1}
|V(t)|
\lesssim \frac{\|{a}\|_{L^\infty([0,T])}}{\sqrt{a_0}} e^{K_{2}T\langle\beta\rangle^{\frac{1}{s}}}|V(0)|.
\end{eqnarray}
If we set, $C_{2}=C^{'}\frac{\|{a}\|_{L^\infty([0,T])}}{\sqrt{a_0}}$, for some absolute constant $C^{'}>0$, then by definition of $V(t)$, we obtain (\ref{odecase2}), that is,
$$|v(t)|^2+|v^{\prime}(t)|^2\leq  C_{2}e^{K_{2}T\langle\beta\rangle^{\frac{1}{s}}}( |v_0|^2+|v_1|^2),$$
finishing the proof of Case 2.

\medskip
\textbf{Case 3:} $a\in\mathcal{C}^{l}([0,T]),$ with $l\geq 2,$ $a(t)\geq 0.$

\smallskip
Consider the quasi-symmetriser of $A(t),$ that is, a family of coercive, Hermitian matrices of the form
\begin{equation}\label{EQ:Q2eps}
Q^{(2)}_{\epsilon}(t):= S(t)+\epsilon^{2}\left(\begin{matrix}1&0\\
0&0\end{matrix}\right)=\left(\begin{matrix}
a(t)&0\\
0&1
\end{matrix} \right)+ \epsilon^2\left(\begin{matrix}1&0\\
0&0\end{matrix}\right),
\end{equation} 
for all $\epsilon\in(0,1],$ so that $(Q^{(2)}_{\epsilon}A-A^{\ast}Q^{(2)}_{\epsilon})$ 
goes to zero as $\epsilon$ goes to zero. The associated perturbed energy will be given by
$$E_{\epsilon}(t):=\left(Q^{(2)}_{\epsilon}(t)V(t), V(t)\right).$$
We proceed by estimating this energy. We have
\begin{eqnarray}
\frac{d}{dt}E_{\epsilon}(t)&=&\left(\frac{d}{dt}Q^{(2)}_{\epsilon}(t)V(t),V(t)\right)+\left(Q^{(2)}_{\epsilon}(t)V_{t}(t),V(t)\right)+\left(Q^{(2)}_{\epsilon}(t)V(t),V_{t}(t)\right)\nonumber\\
&=&\left(\frac{d}{dt}Q^{(2)}_{\epsilon}(t)V(t),V(t)\right)+i\beta^{2}\left(Q^{(2)}_{\epsilon}(t)A(t)V(t),V(t)\right)+\nonumber\\
&& i\left(Q^{(2)}_{\epsilon}(t)B(t)V(t),V(t)\right)-i\beta^{2}\left(Q^{(2)}_{\epsilon}(t)V(t),A(t)V(t)\right)-\nonumber\\
&&i\left(Q^{(2)}_{\epsilon}(t)V(t),B(t)V(t)\right)\nonumber\\
&=&\left(\frac{d}{dt}Q^{(2)}_{\epsilon}(t)V(t),V(t)\right)+i\beta^{2}\left((Q^{(2)}_{\epsilon}(t)A(t)-A^{\ast}(t)Q^{(2)}_{\epsilon}(t))V(t),V(t)\right)+\nonumber\\
&& i\left((Q^{(2)}_{\epsilon}(t)B(t)-B^{\ast}(t)Q^{(2)}_{\epsilon}(t))V(t),V(t)\right).
\end{eqnarray}

Now we will estimate the above terms as follows: 
\begin{enumerate}
	\item $i\left((Q^{(2)}_{\epsilon}(t)A(t)-A^{\ast}(t)Q^{(2)}_{\epsilon}(t))V(t),V(t)\right)\leq \epsilon\left(Q^{(2)}_{\epsilon}(t)V(t), V(t)\right),$
			\item\begin{multline}
			i\left((Q^{(2)}_{\epsilon}(t)B(t)-B^{\ast}(t)Q^{(2)}_{\epsilon}(t))V(t),V(t)\right)\leq\\ \langle\beta\rangle^{2}\left(\|a\|_{L^\infty([0,T])}+\epsilon^{2}+\|b\|_{L^\infty([0,T])}\right)|V(t)|^{2}.
	\end{multline}
\end{enumerate}
Indeed, we will deal with these two terms as follows:
\begin{enumerate}
\item Since $Q^{(2)}_{\epsilon}(t)A(t)-A^{\ast}(t)Q^{(2)}_{\epsilon}(t)=\epsilon^{2}{\tiny\left(\begin{matrix}0&1\\
	-1&0\end{matrix}\right)}$, simple calculations will give
\begin{equation}\label{c3:eqq2}
	\left(Q^{(2)}_{\epsilon}(t)V(t), V(t)\right)=(a(t)+\epsilon^2)|v_{1}|^2+|v_{2}|^2.
\end{equation} 
Now we have	
\begin{eqnarray}\label{c3:eqq4}
i\left((Q^{(2)}_{\epsilon}(t)A(t)-A^{\ast}(t)Q^{(2)}_{\epsilon}(t))V(t),V(t)\right)
&=&i\epsilon^{2}(v_{2}\bar{v_{1}}-v_{1}\bar{v_{2}})\nonumber\\
&=&i\epsilon^{2}2i\text{Im}(v_{2}\bar{v_{1}})\nonumber\\
&\leq& \epsilon2|\epsilon v_{1}||v_{2}|\nonumber\\
&\leq& \epsilon (\epsilon^{2}|v_{1}|^{2}+|v_{2}|^{2})\nonumber\\
&\leq&\epsilon  ((a(t)+\epsilon^{2})|v_{1}|^{2}+|v_{2}|^{2}).
\end{eqnarray}
From \eqref{c3:eqq2} and \eqref{c3:eqq4}, we get
\begin{eqnarray}
i\left((Q^{(2)}_{\epsilon}(t)A(t)-A^{\ast}(t)Q^{(2)}_{\epsilon}(t))V(t),V(t)\right)\leq \epsilon\left(Q^{(2)}_{\epsilon}(t)V(t), V(t)\right).\nonumber
\end{eqnarray}
\item Since $Q^{(2)}_{\epsilon}(t)B(t)-B^{\ast}(t)Q^{(2)}_{\epsilon}(t)=
\left((1-\beta^{2})(a(t)+\epsilon^{2})-b(t)\right)\left(\begin{matrix}0&1\\
	-1&0
	\end{matrix}\right)$, this implies
\begin{multline}\label{c3:eqq5}
i\left((Q^{(2)}_{\epsilon}(t)B(t)-B^{\ast}(t)Q^{(2)}_{\epsilon}(t))V(t),V(t)\right)\\\
\leq	\langle\beta\rangle^{2}\left(\|a\|_{L^\infty([0,T])}+\epsilon^{2}+\|b\|_{L^\infty([0,T])}\right)|V(t)|^{2}.
\end{multline}
\end{enumerate}
Next we estimate the perturbed energy. From \eqref{c3:eqq2}, and for $\epsilon\leq1,$ we  get
\begin{eqnarray}\label{estch1}
\left(Q^{(2)}_{\epsilon}(t)V(t), V(t)\right)&=&(a(t)+\epsilon^2)|v_1|^2+|v_2|^2\nonumber\\
&\leq& (\|a\|_{L^{\infty}([0,T])}+1)(|v_{1}|^{2}+|v_{2}|^{2})\nonumber\\
&=& (\|a\|_{L^{\infty}([0,T])}+1)|V(t)|^{2},
\end{eqnarray}
and
\begin{eqnarray}\label{estch2}
\left(Q^{(2)}_{\epsilon}(t)V(t), V(t)\right)&=&(a(t)+\epsilon^2)|v_1|^2+|v_2|^2\nonumber\\
&\geq& \epsilon^{2}|v_{1}|^{2}+\frac{\epsilon^{2}}{\|a\|_{L^{\infty}([0,T])}+1}|v_{2}|^{2}\nonumber\\
&\geq&\frac{\epsilon^{2}}{\|a\|_{L^{\infty}([0,T])}+1}|v_{1}|^{2}+\frac{\epsilon^{2}}{\|a\|_{L^{\infty}([0,T])}+1}|v_{2}|^{2}\nonumber\\
&=&c_{1}^{-1}\epsilon^{2}|V(t)|^{2}.
\end{eqnarray}
From \eqref{estch1} and \eqref{estch2}, we get
\begin{equation}\label{estch3}
c^{-1}_{1}\epsilon^{2}|V(t)|^{2}\leq \left(Q^{(2)}_{\epsilon}(t)V(t), V(t)\right)\leq c_1|V(t)|^{2},
\end{equation}
where $c_{1}=\|a\|_{L^{\infty}([0,T])}+1.$\\

 Using the above estimates, we have
\begin{eqnarray}\label{c3:eq1}
\frac{d}{dt}E_{\epsilon}(t)&\leq&\left(\frac{d}{dt}Q^{(2)}_{\epsilon}(t)V(t),V(t)\right)+\langle\beta\rangle^{2}\epsilon\left(Q^{(2)}_{\epsilon}(t)V(t), V(t)\right)+\nonumber\\
&&\langle\beta\rangle^{2}\left(\|a\|_{L^\infty([0,T])}+\epsilon^{2}+
\|b\|_{L^\infty([0,T])}\right)|V(t)|^{2}\nonumber\\
&\leq&\left(\frac{d}{dt}Q^{(2)}_{\epsilon}(t)V(t), V(t)\right)+\langle\beta\rangle^{2}\epsilon E_{\epsilon}(t)+\nonumber\\
&&c_1\epsilon^{-2}\langle\beta\rangle^{2}\left(\|a\|_{L^\infty([0,T])}+\epsilon^{2}+
\|b\|_{L^\infty([0,T])}\right)E_{\epsilon}(t)\nonumber\\
&=&\left(\frac{d}{dt}Q^{(2)}_{\epsilon}(t)V(t), V(t)\right)+\langle\beta\rangle^{2}\epsilon E_{\epsilon}(t)
+c_1\epsilon^{-2}\langle\beta\rangle^{2}\|a\|_{L^\infty([0,T])}E_{\epsilon}(t)+\nonumber\\
&&c_{1}\langle\beta\rangle^{2}E_{\epsilon}(t)+c_1\epsilon^{-2}\langle\beta\rangle^{2}\|b\|_{L^\infty([0,T])}E_{\epsilon}(t),
\end{eqnarray}
which gives
\begin{multline}\label{c3:mains}
	\frac{d}{dt}E_{\epsilon}(t)
	\leq \Bigg[ \frac{\left(\frac{d}{dt}Q^{(2)}_{\epsilon}(t)V(t), V(t)\right)}{\left(Q^{(2)}_{\epsilon}(t)V(t), V(t)\right)}+\langle\beta\rangle^{2}\epsilon+c_1\epsilon^{-2}\langle\beta\rangle^{2}(\|a\|_{L^\infty([0,T])}+\|b\|_{L^{\infty}([0,T])})\Bigg.\\
	\Bigg.+ c_{1}\langle\beta\rangle^{2}\Bigg] E_{\epsilon}(t),
\end{multline}
where $c_{1}=\|a\|_{L^{\infty}([0,T])}+1.$ We will use the following estimate:
\begin{equation}\label{ChRuest}
\int_{0}^{T}\frac{\left(\frac{d}{dt}Q^{(2)}_{\epsilon}(t)V(t), V(t)\right)}{\left(Q^{(2)}_{\epsilon}(t)V(t), V(t)\right)}dt\leq c_{T}c_{1}\epsilon^{-\frac{2}{l}}\|a\|_{\mathcal{C}^{l}([0,T])}^{\frac{1}{l}},
\end{equation}
where $c_{1}=\|a\|_{L^{\infty}([0,T])}+1$ for some $c_T>0$, depending only on $T$, which we assume for a moment. From \eqref{c3:mains} and \eqref{ChRuest}, we get
\begin{eqnarray}\label{c3:eq2}
	&&\frac{d}{dt}E_{\epsilon}(t)\nonumber\\&\leq& \left( c_{T}c_{1}\epsilon^{-\frac{2}{l}}\|a\|_{\mathcal{C}^{l}([0,T])}^{\frac{1}{l}}+\langle\beta\rangle^{2}\epsilon+
	c_1\epsilon^{-2}\langle\beta\rangle^{2}(\|a\|_{L^\infty([0,T])}+\|b\|_{L^{\infty}([0,T])})+c_{1}\langle\beta\rangle^{2}\right) E_{\epsilon}(t)\nonumber\\
	&=&\kappa E_{\epsilon}(t),
\end{eqnarray} 
where $\kappa= c_{T}c_{1}\epsilon^{-\frac{2}{l}}\|a\|_{\mathcal{C}^{l}([0,T])}^{\frac{1}{l}}+\langle\beta\rangle^{2}\epsilon+
c_1\epsilon^{-2}\langle\beta\rangle^{2}(\|a\|_{L^\infty([0,T])}+\|b\|_{L^{\infty}([0,T])})+c_{1}\langle\beta\rangle^{2}$.
Applying the Gronwall lemma on \eqref{c3:eq2}, we deduce that
\begin{multline}
	E_{\epsilon}(t)\leq e^{\kappa T}E_{\epsilon}(0)\\ =E_{\epsilon}(0)e^{\left(c_{T}c_{1}\epsilon^{-\frac{2}{l}}\|a\|_{\mathcal{C}^{l}([0,T])}^{\frac{1}{l}}+\langle\beta\rangle^{2}\epsilon+
		c_1\epsilon^{-2}\langle\beta\rangle^{2}(\|a\|_{L^\infty([0,T])}+\|b\|_{L^{\infty}([0,T])})+c_{1}\langle\beta\rangle^{2}\right)T}.
\end{multline}
Combining this inequality with \eqref{estch3}, we get
\begin{multline}
c^{-1}_{1}\epsilon^{2}|V(t)|^{2}\leq E_{\epsilon}(t)
\leq\\ c_1 
e^{\left(c_{T}c_{1}\epsilon^{-\frac{2}{l}}\|a\|_{\mathcal{C}^{l}([0,T])}^{\frac{1}{l}}+\langle\beta\rangle^{2}\epsilon+
	c_1\epsilon^{-2}\langle\beta\rangle^{2}(\|a\|_{L^\infty([0,T])}+\|b\|_{L^{\infty}([0,T])})+c_{1}\langle\beta\rangle^{2}\right)T}|V(0)|^2.	
\end{multline}
Choosing, $\epsilon^{-\frac{2}{l}}=\langle\beta\rangle^{2}\epsilon$ and setting $\sigma=1+\frac{l}{2}$, we get $\epsilon=\langle\beta\rangle^{-\frac{l}{\sigma}}\leq 1$ for all $\beta>0$ and
$\epsilon^{-\frac{2}{l}}=\langle\beta\rangle^{\frac{2}{\sigma}}$, so that
\begin{eqnarray}\label{c3:eqn1}
&&|V(t)|^{2}\nonumber\\
&\leq& c_1^2 \langle\beta\rangle^{\frac{2l}{\sigma}}e^{\left((c_{T}c_{1}\|a\|_{\mathcal{C}^{l}([0,T])}^{\frac{1}{l}}+1)\langle\beta\rangle^{\frac{2}{\sigma}}+c_{1}\langle\beta\rangle^{2+\frac{2l}{\sigma}}(\|a\|_{L^{\infty}([0,T])} +\|b\|_{L^{\infty}([0,T])})
	+c_{1}\langle\beta\rangle^{2}\right)T}|V(0)|^{2}\nonumber\\
&=&c_{1}^{2}\langle\beta\rangle^{4-\frac{4}{\sigma}}e^{\left((c_{T}c_{1}\|a\|_{\mathcal{C}^{l}([0,T])}^{\frac{1}{l}}+1)\langle\beta\rangle^{\frac{2}{\sigma}}+c_{1}\langle\beta\rangle^{6-\frac{4}{\sigma}}(\|a\|_{L^{\infty}([0,T])} +\|b\|_{L^{\infty}([0,T])})
	+c_{1}\langle\beta\rangle^{2}\right)T}|V(0)|^{2}\nonumber\\
&\leq& c_{1}^{2}e^{\left(1+c_{1}\left(c_{T}\|a\|_{\mathcal{C}^{l}([0,T])}^{\frac{1}{l}}+c_{1}^{-1}+\|a\|_{L^{\infty}([0,T])} +\|b\|_{L^{\infty}([0,T])}
	+1\right)T\right)\langle\beta\rangle^{6-\frac{4}{\sigma}}}|V(0)|^{2}\nonumber\\
&=&C_{3}e^{K_{3}\langle\beta\rangle^{6-\frac{4}{\sigma}}}|V(0)|^{2}.
\end{eqnarray}
If we set  $K_{3}=1+c_{1}\left(c_{T}\|a\|_{\mathcal{C}^{l}([0,T])}^{\frac{1}{l}}+c_{1}^{-1}+\|a\|_{L^{\infty}([0,T])} +\|b\|_{L^{\infty}([0,T])}
+1\right)T$ and $C_{3}=c_{1}^{2}$, then by the definition of $V(t)$ we get 

\begin{equation}
|v(t)|^{2}+|v'(t)|^{2}\leq C_{3}e^{K_{3}\langle\beta\rangle^{6-\frac{4}{\sigma}}}
	(|v_{0}|^{2}+|v_{1}|^{2}).	
\end{equation}
This gives \eqref{odecase3}. It still remains to show \eqref{ChRuest}. We can obtain this estimate by an argument from \cite{DS}.
Indeed, using \eqref{EQ:Q2eps}, for $V=(V_1, V_2)^T$, one readily obtains
\begin{equation}\label{EQ:Q1}
|\left(Q_{\epsilon}^{(2)}(t)V,V\right)|=a(t)|V_1|^2+|V_2|^2+\epsilon^2|V_1|^2,
\end{equation}
and
\begin{equation}\label{EQ:Q2}
\left(\frac{d}{dt}Q^{(2)}_{\epsilon}(t)V,V\right)=a^{\prime}(t)|V_1|^2.
\end{equation}
From \eqref{EQ:Q1} we have
\begin{equation}\label{estch}
c^{-1}_{1}\epsilon^{2}|V(t)|^{2}\leq \left(Q^{(2)}_{\epsilon}(t)V(t), V(t)\right)\leq c_1|V(t)|^{2},
\end{equation}
where $c_{1}=\|a\|_{L^{\infty}([0,T])}+1.$ At the same time
\begin{eqnarray}\label{c3:eqqq1}
\int_{0}^{T}\frac{\left(\frac{d}{dt}Q^{(2)}_{\epsilon}(t)V(t), V(t)\right)}{\left(Q^{(2)}_{\epsilon}(t)V(t), V(t)\right)}dt&=&\int_{0}^{T}\frac{\left(\frac{d}{dt}Q^{(2)}_{\epsilon}(t)V(t), V(t)\right)}{\left(Q^{(2)}_{\epsilon}(t)V(t), V(t)\right)^{1-\frac{1}{l}}\left(Q^{(2)}_{\epsilon}(t)V(t), V(t)\right)^{\frac{1}{l}}}dt\nonumber\\
&\leq&\int_{0}^{T}\frac{\left(\frac{d}{dt}Q^{(2)}_{\epsilon}(t)V(t), V(t)\right)}{\left(Q^{(2)}_{\epsilon}(t)V(t), V(t)\right)^{1-\frac{1}{l}}(c_1^{-1}\epsilon^{2})^{\frac{1}{l}}|V(t)|^{\frac{2}{l}}}dt\nonumber\\
&\leq& c_1\epsilon^{-\frac{2}{l}}\int_{0}^{T}\frac{\left|\left(\frac{d}{dt}Q^{(2)}_{\epsilon}(t)V(t), V(t)\right)\right|}{\left(Q^{(2)}_{\epsilon}(t)V(t), V(t)\right)^{1-\frac{1}{l}}|V(t)|^{\frac{2}{l}}}dt.
\end{eqnarray}
Next we try to find the estimate of
\begin{equation}\label{est1}
\int_{0}^{T}\frac{\left|\left(\frac{d}{dt}Q^{(2)}_{\epsilon}(t)V(t), V(t)\right)\right|}{\left(Q^{(2)}_{\epsilon}(t)V(t), V(t)\right)^{1-\frac{1}{l}}|V(t)|^{\frac{2}{l}}}dt.
\end{equation}
The above case \eqref{est1}, is a special case of a general estimate
\begin{equation}\label{refest}
\int^{T}_{0}\frac{|f^{\prime}(t)|}{|f(t)|^{1-\frac{1}{l}}}dt\leq c_{T}\|f\|^{\frac{1}{l}}_{\mathcal{C}^{l}([0,T])},
\end{equation}
which holds for any real or complex-valued function $f\in\mathcal{C}^{l}([0,T]).$ Since $Q_{\epsilon}^{(2)}(t)$ is a diagonal matrix, an estimate for \eqref{est1} can be obtained directly by applying \eqref{refest} to each of the entries of $Q_{\epsilon}^{(2)}(t)$.
Then noting that $|V|^2=|V_1|^2+|V_2|^2,$ we get
\begin{eqnarray}\label{est2}
\int_{0}^{T}\frac{|\left(\frac{d}{dt}Q^{(2)}_{\epsilon}(t)V(t), V(t)\right)|}{\left(Q^{(2)}_{\epsilon}(t)V(t), V(t)\right)^{1-\frac{1}{l}}|V(t)|^{\frac{2}{l}}}dt&\leq& 
\int_{0}^{T}\frac{a^{\prime}(t)|V_1(t)|^2}{(a(t)+\epsilon^2)|V_1|^2)^{1-\frac{1}{l}}|V(t)
|^{\frac{2}{l}}}dt \nonumber\\
&\leq& \int_{0}^{T} \frac{|a^{\prime}(t)| |V_1(t)|^2}{({a(t)+\epsilon^2})^{1-\frac{1}{l}}|V_1(t)|^{2(1-\frac{1}{l})}|V_1(t)|^{\frac{2}{l}}}dt\nonumber\\
&\leq& \int_{0}^{T} \frac{|a^{\prime}(t)|}{({a(t)+\epsilon^2})^{1-\frac{1}{l}}}dt\nonumber\\
&\leq& \int_{0}^{T} \frac{|a^{\prime}(t)|}{{|a(t)|}^{1-\frac{1}{l}}}dt\nonumber\\
&\leq& c_{T}\|a\|_{\mathcal{C}^{l}([0,T])}^{\frac{1}{l}}.
\end{eqnarray} 
Using \eqref{est2} in \eqref{c3:eqqq1} we get
\begin{equation}\label{c3:main}
\int_{0}^{T}\frac{\left(\frac{d}{dt}Q^{(2)}_{\epsilon}(t)V(t), V(t)\right)}{\left(Q^{(2)}_{\epsilon}(t)V(t), V(t)\right)}dt \leq c_1 c_{T}\epsilon^{-\frac{2}{l}}\|a\|_{\mathcal{C}^{l}([0,T])}^{\frac{1}{l}},
\end{equation}
yielding \eqref{ChRuest} where $c_{1}=\|a\|_{L^{\infty}([0,T])}+1$.

\medskip
\textbf{Case 4: $a\in \mathcal{C}^{\alpha}([0,T])$, with $0<\alpha<2$, $a(t)\geq 0.$}

\smallskip
In this last case we extend the proof of Case 2. However, under these assumptions the eigenvalues of the matrix $A(t)$, that are $\pm\sqrt{a(t)}$ might coincide, and hence they are H\"{o}lder continuous of order $\frac{\alpha}2$ instead of $\alpha.$ 
In order to adapt this proof to the one for Case 2, and to simplify the notation, we assume without loss of generality that $a\in\mathcal{C}^{2\alpha}([0,T])$ with $0<\alpha<1,$ so that $\sqrt{a}\in\mathcal{C}^{\alpha}([0,T]).$ 
Then we change $\alpha$ into $\frac\alpha{2}$ in the final statement.

We look again for solutions of the form
$$V(t)=e^{-\rho(t)\langle\beta\rangle^{\frac{1}{s}}}(\det H(t))^{-1}H(t)W(t),$$ with the real valued function $\rho(t),$ the exponent $s$ and the energy $W(t)$ to be chosen later, while $H(t)$ is the matrix given by
$$H(t)=\left( \begin{matrix}
1&1\\
\lambda_{1,\alpha}^{\epsilon}(t)&\lambda_{2,\alpha}^{\epsilon}(t)
\end{matrix}   \right),$$
where the regularised eigenvalues of $A(t)$ are $\lambda_{1,\alpha}^{\epsilon}(t)$ and $\lambda_{2,\alpha}^{\epsilon}(t)$ differ from the ones defined in the previous case in the following way,
$$\lambda_{1,\alpha}^{\epsilon}(t):=\left(-\sqrt{a}\ast\phi_{\epsilon}\right)(t)+\epsilon^{\alpha},$$
$$\lambda_{2,\alpha}^{\epsilon}(t):=\left(+\sqrt{a}\ast\phi_{\epsilon}\right)(t)+2\epsilon^{\alpha}.$$
Arguing as in Case 2, we have 
\begin{equation}{\label{c1}}
 \det H(t)=\lambda_{2,\alpha}^{\epsilon}(t)-\lambda_{1,\alpha}^{\epsilon}(t)\geq \epsilon^{\alpha},
\end{equation}
 and also, similar to \eqref{lambes1}, we have
 \begin{eqnarray}
 \left|\lambda_{1,\alpha}^{\epsilon}(t)+\sqrt{a(t)}\right|&=&\left|\left(-\sqrt{a}\ast\phi_{\epsilon}\right)(t)+\sqrt{a(t)}+\epsilon^{\alpha}\right|\nonumber\\
 &=&\left|\int_{\mathbb{R}}\sqrt{a(t-y)}\phi_{\epsilon}(y)dy-\sqrt{a(t)}\int_{\mathbb{R}}\phi(y)dy-\epsilon^{\alpha}\right|\nonumber\\
 &=&\left|\int_{\mathbb{R}}\left(\sqrt{a(t-\epsilon x)}-\sqrt{a(t)}\right)\phi(x)dx-\epsilon^{\alpha}\right|\nonumber\\
 &\leq&\left(\|\sqrt{a}\|_{ {\mathcal C}^{\alpha}([0,T])}+1\right)\epsilon^{\alpha}.
 \end{eqnarray}
 Similarly, we can compute $\left|\lambda_{2,\alpha}^{\epsilon}(t)-\sqrt{a(t)}\right|$, and we get 
\begin{equation}{\label{c2}} \left|\lambda^{\epsilon}_{1,\alpha}(t)+\sqrt{a(t)}\right|\leq c_{1}\epsilon^{\alpha},\end{equation}
 \begin{equation}{\label{c3}}
 \left|\lambda^{\epsilon}_{2,\alpha}(t)-\sqrt{a(t)}\right|\leq c_{2}\epsilon^{\alpha},
\end{equation}
with $c_{1}= \|\sqrt{a}\|_{ {\mathcal C}^{\alpha}([0,T])}+1$ and $c_{2}= \|\sqrt{a}\|_{ {\mathcal C}^{\alpha}([0,T])}+2.$
Now we look at the energy estimates:  
\begin{align}{\label{estimate1}}
\frac{d}{dt}|W(t)|^2&\leq  (2\rho^{\prime}(t)\langle\beta\rangle^{\frac{1}{s}}+2 \| H^{-1}(t)H_{t}(t) \|+2|\left(\det H(t)\right)^{-1}\left(\det H\right)_{t}(t)|+\nonumber\\
&\langle\beta\rangle^{2} \| H^{-1}AH-(H^{-1}AH)^{\ast}\|+\| H^{-1}BH-(H^{-1}BH)^{\ast}\|)\|W\|^{2}.
\end{align}
In the present setting, inequality \eqref{determinantes} is replaced by
\begin{equation}\label{determinantes2}
\det H(t)=\lambda_{2}^{\epsilon}(t)-\lambda_{1}^{\epsilon}(t)\geq  \epsilon^\alpha.
\end{equation} 
Taking this into account, the estimates (1) and (2) in Case 2 are replaced by
\begin{enumerate}
\item $\|H^{-1}(t)H_{t}(t)\|\lesssim \|\sqrt{a}\|_{{\mathcal C}^{\alpha}([0,T])}\epsilon^{-1}$,
\item $|\left(\det H(t)\right)^{-1}\left(\det H\right)_{t}(t)|\lesssim\|\sqrt{a}\|_{{\mathcal C}^{\alpha}([0,T])}\epsilon^{-1}$.
\end{enumerate}

We now will estimate $\|(H^{-1}AH)-(H^{-1}AH)^{\ast}\|.$ First, we  explicitly write this matrix, recalling \eqref{EQ:Hs}, that is 
$$H^{-1}AH-(H^{-1}AH)^{\ast}=\frac{1}{\lambda_{2,\alpha}^{\epsilon}-\lambda_{1,\alpha}^{\epsilon}}\left(\begin{matrix}
0 & -2a+(\lambda_{2,\alpha}^{\epsilon})^2+(\lambda_{1,\alpha}^{\epsilon})^2\\
2a-(\lambda_{2,\alpha}^{\epsilon})^2-(\lambda_{1,\alpha}^{\epsilon})^2&0
\end{matrix}\right).$$
We consider the functions $\left|a(t)-(\lambda_{2,\alpha}^{\epsilon})^2\right|$ and $\left|a(t)-(\lambda_{1,\alpha}^{\epsilon})^2\right|.$
A straightforward estimate, using \eqref{c2} and \eqref{c3}, gives
\begin{eqnarray}
\left|a(t)-(\lambda_{2,\alpha}^{\epsilon})^{2}\right|&=&\left|(\sqrt{a(t)}-\lambda_{2,\alpha}^{\epsilon})(\sqrt{a(t)}+\lambda_{2,\alpha}^{\epsilon})\right|\nonumber\\
&\leq& c_{2}\epsilon^{\alpha}\left|(\sqrt{a(t)}+\lambda_{2,\alpha}^{\epsilon})\right|\nonumber\\
&=& c_{2}\epsilon^{\alpha}\left|\sqrt{a(t)}+\int_{\mathbb{R}}\sqrt{a(t-s\epsilon)}\phi(s)ds+2\epsilon^{\alpha}\right|\nonumber\\
&=&c_{2}\epsilon^{\alpha}\left|\int_{\mathbb{R}}\left(\sqrt{a(t)}-\sqrt{a(t+s\epsilon)}\right)\phi(s)ds+2\epsilon^{\alpha}\right|\nonumber\\
&\leq& c_{2}(2+\|\sqrt{a}\|_{{\mathcal C}^\alpha([0,T])})\epsilon^{2\alpha}\nonumber\\
&=&(2+\|\sqrt{a}\|_{{\mathcal C}^\alpha([0,T])})^{2}\epsilon^{2\alpha}.
\end{eqnarray}
Similarly, we have
\begin{eqnarray}
\left|a(t)-(\lambda_{1,\alpha}^{\epsilon})^2\right|
&\leq&(1+\|\sqrt{a}\|_{{\mathcal C}^\alpha([0,T])})^{2}\epsilon^{2\alpha}.
\end{eqnarray}
 Combining, the above estimates, we have
\begin{equation}\label{K3}
\left|a(t)-(\lambda_{i,\alpha}^{\epsilon})^2\right|\leq (2+\|\sqrt{a}\|_{ {\mathcal C}^{\alpha}([0,T])})^{2}\epsilon^{2\alpha}.
\end{equation}
It follows using \eqref{determinantes2}, that 
$$
\|(H^{-1}AH)-(H^{-1}AH)^{\ast}\|\lesssim (2+\|\sqrt{a}\|_{ {\mathcal C}^{\alpha}([0,T])})^{2}\epsilon^{\alpha}.$$
Now, we will estimate $\|H^{-1}BH-(H^{-1}BH)^{\ast}\|.$ Consider the matrix,
\begin{multline}\label{c4:eq4}
H^{-1}BH-(H^{-1}BH)^{\ast}=\\
{\tiny \frac{1}{\lambda^{\epsilon}_{2,\alpha}-\lambda^{\epsilon}_{1,\alpha}}
\left(\begin{matrix}
0&-2b+(1-\beta^{2})\left((\lambda_{1,\alpha}^{\epsilon})^{2}+(\lambda_{2,\alpha}^{\epsilon})^{2}\right)\\
2b-(1-\beta^{2})\left((\lambda_{1,\alpha}^{\epsilon})^{2}+(\lambda_{2,\alpha}^{\epsilon})^{2}\right)&0
\end{matrix}\right)}.
\end{multline} 
Since, $\lambda_{2,\alpha}^{\epsilon}+\lambda_{1,\alpha}^{\epsilon}=3\epsilon^{\alpha}$ and $\lambda_{2,\alpha}^{\epsilon}-\lambda_{1,\alpha}^{\epsilon}\geq\epsilon^{\alpha}$, this implies $\left(\lambda_{2,\alpha}^{\epsilon}\right)^{2}\geq \left(\lambda_{1,\alpha}^{\epsilon}\right)^{2}$, and 
\begin{eqnarray}\label{c4:eq5}
\left(\lambda_{1,\alpha}^{\epsilon}\right)^{2}+\left(\lambda_{2,\alpha}^{\epsilon}\right)^{2}&\leq&2\left(\lambda_{2,\alpha}^{\epsilon}\right)^{2}\nonumber\\
&=&2\left(\left(+\sqrt{a}\ast\phi_{\epsilon}\right)(t)+2\epsilon^{\alpha}\right)^{2}\nonumber\\
&=&2\left(\int\sqrt{a(t-s\epsilon)}\phi(s)ds+2\epsilon^{\alpha}\right)^{2}\nonumber\\
&\leq&2\left(\sqrt{a(t)}+\int\sqrt{a(t-s\epsilon)}\phi(s)ds+2\epsilon^{\alpha}\right)^{2}\nonumber\\
&=&2\left(\int\left(\sqrt{a(t)}-\sqrt{a(t+s\epsilon)}\right)\phi(s)ds+2\epsilon^{\alpha}\right)^{2}\nonumber\\
&\leq&2\left(\|\sqrt{a}\|_{{\mathcal C}^\alpha([0,T])}\epsilon^{\alpha}+2\epsilon^{\alpha}\right)^{2}\nonumber\\
&=&2\epsilon^{2\alpha}\left(\|\sqrt{a}\|_{{\mathcal C}^\alpha([0,T])}+2\right)^{2}.\nonumber\\
\end{eqnarray}
Using \eqref{c4:eq4} and \eqref{c4:eq5}, we can estimate
\begin{eqnarray}\label{EQ:estbs2}
\|H^{-1}BH-(H^{-1}BH)^{\ast}\|&\leq& 2\left(\langle\beta\rangle^{2}\left(\|\sqrt{a}\|_{{\mathcal C}^\alpha([0,T])}+2\right)^{2}\epsilon^{\alpha}+\|b\|_{L^\infty([0,T])}\epsilon^{-\alpha}\right)\nonumber\\
&\lesssim& \langle\beta\rangle^{2}\left(\|\sqrt{a}\|_{{\mathcal C}^\alpha([0,T])}+2\right)^{2}\epsilon^{\alpha}+\|b\|_{L^\infty([0,T])}\epsilon^{-\alpha}.
\end{eqnarray} 
Using these estimates in \eqref{estimate1}, we get
\begin{multline}{\label{estimate2}}
\frac{d}{dt}|W(t)|^2
\lesssim ( 2\rho^{\prime}(t)\langle\beta\rangle^{\frac{1}{s}}+ 2\|\sqrt{a}\|_{{\mathcal C}^{\alpha}([0,T])}\epsilon^{-1}
+2\langle\beta\rangle^{2} (2+\|\sqrt{a}\|_{ {\mathcal C}^{\alpha}([0,T])})^{2}\epsilon^{\alpha}\nonumber\\
 +\|b\|_{L^\infty([0,T])}\epsilon^{-\alpha})\|W\|^{2}.\nonumber
\end{multline}
We choose $\epsilon^{-1}=\langle\beta\rangle^{2}\epsilon^{\alpha}$, which yields $\epsilon=\langle\beta\rangle^{\frac{-2}{1+\alpha}}\leq 1$, for all $\beta>0$. 
Let us now define $\rho(t):=\rho(0)-K_{4}T$ with $K_{4}>0$, to be chosen later, which gives
\begin{multline}
\frac{d}{dt}|W(t)|^2\lesssim( -2K_{4}\langle\beta\rangle^{\frac{1}{s}}+ \|\sqrt{a}\|_{{\mathcal C}^{\alpha}([0,T])}\langle\beta\rangle^{2\gamma}
+(2+\|\sqrt{a}\|_{ {\mathcal C}^{\alpha}([0,T])})^{2}\langle\beta\rangle^{2\gamma}\\
+\|b\|_{L^\infty([0,T])}\langle\beta\rangle^{2(1-\gamma)})\|W\|^{2},
\end{multline}
where $\gamma=\frac{1}{\alpha+1}$. Since $0<\alpha<1$, this implies $\frac{1}{2}<\gamma<1$. Therefore  
\begin{equation}\label{c4:eq1}
	\frac{d}{dt}|W(t)|^2
	\lesssim
	\left(-2K_{4}\langle\beta\rangle^{\frac{1}{s}}
	+  \kappa\langle\beta\rangle^{2\gamma}
	\right)|W(t)|^2,	
\end{equation}
where $\kappa=\|\sqrt{a}\|_{{\mathcal C}^{\alpha}([0,T])}+(2+\|\sqrt{a}\|_{ {\mathcal C}^{\alpha}([0,T])})^{2}
+\|b\|_{L^\infty([0,T])}$.\\

 If we choose $K_{4}=\frac{\kappa}{2}$ and $\frac{1}{s}\geq2\gamma$=$\frac{2}{\alpha+1}$, then
for all $t\in[0,T]$ and $\beta>0$, we have 
$$\frac{d}{dt}|W(t)|^2\leq 0.$$
Then similar to the Case 2, this monotonicity of the energy $W(t)$  yields the  bound of the solution vector $V(t)$ as:
\begin{equation}
|V(t)|\leq {\|H(t)\|}{\|H(0)\|^{-1}\left|\det H(t)\right|^{-1}}\left|\det H(0)\right|e^{K_{4}T\langle\beta\rangle^{\frac{1}{s}}}|V(0)|.
\end{equation}
Using \eqref{determinantes2}, we have
$${\|H(t)\|}{\|H(0)\|^{-1}\left|\det H(t)\right|^{-1}}\left|\det H(0)\right|\lesssim
(2+ \|\sqrt{a}\|_{ {\mathcal C}^{\alpha}([0,T])})^2
.$$
So we have the inequality
\begin{eqnarray}\label{c4:eqn1}
|V(t)|\lesssim (2+ \|\sqrt{a}\|_{ {\mathcal C}^{\alpha}([0,T])})^2
e^{K_{4}T\langle\beta\rangle^{\frac{1}{s}}}|V(0)|.
\end{eqnarray}
If we replace $\alpha$  by $\frac{\alpha}{2}$ and set  $C_{4}=C^{'}(2+ \|\sqrt{a}\|_{ {\mathcal C}^{\frac{\alpha}{2}}([0,T])})^2$, then by definition of $V(t)$, we get 
\begin{eqnarray}
|v(t)|^2+|v^{\prime}(t)|^2\leq  C_{4}e^{K_{4}T\langle\beta\rangle^{\frac{1}{s}}}( |v_0|^2+|v_1|^2).
\end{eqnarray}
This completes the proof of Case 4, with $0<s\leq \frac{\alpha+2}{4}$.
\end{proof}

By the standard methods of ordinary differential equations, Proposition \ref{odelem}, has the following corollary for the corresponding inhomogeneous equations.
\begin{corollary}\label{eqq11}
	Let $T>0$ and $\beta>0$ be positive constants, let $b=b(t)\geq 0$, and $g(t)$ be bounded real-valued functions in $L^{\infty}([0,T])$, and let $a\not\equiv 0$ be a function satisfying one of the conditions of Proposition \ref{odelem}.
	Consider the following Cauchy problem:
	\begin{equation}\label{eqq12}
		\left\{
		\begin{array}{ll}
			w^{\prime\prime}(t)+\beta^{2}a(t)w(t)+b(t)w(t)=g(t), \quad t\in(0,T],\\
			w(0)=w_{0}\in\mathbb{C},\\
			w^{\prime}(0)=w_1\in\mathbb{C}.
		\end{array}
		\right.
	\end{equation}
	Let $v(t)$ be the solution of \eqref{eqq12}, with $g=0$, that is, it satisfies \eqref{ode}, and the respective estimates in Proposition \ref{odelem}. Then we have the estimate 
	
	\begin{equation}
		|w(t)|^{2}+|w^{\prime}(t)|^{2} \lesssim 
		|v(t)|^{2}+|v^{\prime}(t)|^{2} 
		+
		(1+\beta^{2})^{2}\|g\|^2_{L^{2}([0,T])},
	\end{equation}
for all $t\in[0,T]$, with the constant independent of $\beta$.
\end{corollary} 
\begin{proof}
	Similar to the proof of Proposition \ref{odelem}, we denote
	\begin{equation}
		W(t) := \left(
		\begin{matrix}
			i w(t) \\
			\partial_{t}w(t)
		\end{matrix}\right), 
		~~~~~ W_{0}:=\left(
		\begin{matrix}
			i {w_0}\\
			{ w_1}\end{matrix}\right),
	\end{equation} 
	and we will use the matrices 
	\begin{equation}
		A(t):=\left(\begin{matrix}
			0&1\\
			a(t)&0
		\end{matrix}\right),~~~
		B(t):= \left(\begin{matrix}
			0&1-\beta^{2}\\
			b(t)&0
		\end{matrix}\right)~~{\textrm{ and }}
		~~~G(t):=\left(\begin{matrix}
			0\\
			g(t)
		\end{matrix}\right).
	\end{equation}
	In this notation \eqref{eqq12}, becomes 
	\begin{equation}\label{eqq13}
		\left\{
		\begin{array}{ll}
			W^{\prime}=i\beta^{2} A(t)W(t)+iB(t)W(t)+G(t),\\
			W(0)=W_0.
		\end{array}\right.
	\end{equation}
	Consequently, we consider the following two equations
	\begin{equation}\label{eqq14}\left\{
		\begin{array}{ll}
			V_{I}^{\prime}=i\left(\beta^{2} A(t)+B(t)\right)V_{I}(t),\\
			V_{I}(0)=W_0,
		\end{array}\right.
	\end{equation}
	and
	\begin{equation}\label{eqq15}\left\{
		\begin{array}{ll}
			V_{II}^{\prime}=i\left(\beta^{2} A(t)+B(t)\right)V_{II}(t)+G(t),\\
			V_{II}(0)=0.
		\end{array}\right.
	\end{equation}
	Then the solution of \eqref{eqq13}, is given by $W=V_I+V_{II}.$
	From the general theory of ordinary differential equations  we can say that if
	\begin{equation}
		\Phi(t;s) := \left(
		\begin{matrix}
			\phi_1(t;s) \\
			\phi_{2}(t;s)
		\end{matrix}\right), 
	\end{equation} 
	is  the solution of the equation
	\begin{equation}\label{Fsthode}\left\{
		\begin{array}{ll}
			\Phi^{\prime}(t,s)=i\left(\beta^{2} A(t)+B(t)\right)\Phi(t,s),\quad s\leq t,\\
			\Phi(s;s)=G(s),  
		\end{array}\right.
	\end{equation}
	then 
	$$\Phi(t,s)=e^{i\int_{s}^{t}\left(\beta^{2} A(\tau)+B(\tau)\right)d\tau}G(s).$$
	Then the solution of \eqref{eqq15} is given by 
	$V_{II}(t)=\int_{0}^{t}\Phi(t;s)ds$.
	Consequently, we have
	\begin{equation}\label{c0}
		W(t) = V_{I}+  \int_{0}^{t}e^{i\int_{s}^{t}\left(\beta^{2} A(\tau)+B(\tau)\right)d\tau} G(s)ds.
	\end{equation}
	Here, on the right hand side, we will estimate the matrix term $e^{i\int_{s}^{t}\left(\beta^{2} A(\tau)+B(\tau)\right)d\tau}$. Since  $A(t)=\left(\begin{matrix}
		0&1\\
		a(t)&0
	\end{matrix}\right)$ and $B(t)= \left(\begin{matrix}
		0&1-\beta^2\\
		b(t)&0
	\end{matrix}\right)$, this implies that  $$M(t):=\beta^{2} A(t)+B(t)=\left(\begin{matrix}
		0&1\\
		\beta^{2}a(t)+b(t)&0
	\end{matrix}\right).$$ 
	This gives $e^{i\int_{s}^{t}\left(\beta^{2} A(\tau)+B(\tau)\right)d\tau}=e^{i\tilde{M}}$ where denoting $\tilde{M}:=\int_{s}^{t}M(\tau)d\tau$, we can abbreviate writing
	$\tilde{M}= \left( \begin{matrix}
		0&\tau\\
		m&0
	\end{matrix}\right),$ where $\tau=t-s>0$ and $m=\int\limits_{s}^{t} (\beta^{2}a(\tau)+b(\tau)) d\tau\geq 0,$ since $a(t), b(t)\geq 0$ for all $t\in [0,T]$. For now we will consider the case $m\neq 0$. 	It can be readily seen from the matrix multiplication  and induction that
	$$\tilde{M}^{2k}=(\tau m)^k I$$
	and $$\tilde{M}^{2k+1}=(\tau m)^{k}\tilde{M},$$ for $k=0,1,2,....$
	Then we have
	\begin{eqnarray}\label{estref1}
		e^{i\tilde{M}}&=&\sum_{n=0}^{\infty} \frac{(i\tilde{M})^n}{n!}\nonumber\\
		&=& \sum_{ k=0}^{\infty}\frac{(i\tilde{M})^{2k}}{(2k)!}+ \sum_{k=0}^{\infty}\frac{(i\tilde{M})^{2k+1}}{(2k+1)!}\nonumber\\
		&=& \sum_{ k=0}^{\infty}\frac{i^{2k}(\tau m)^{k}}{(2k)!}I + \sum_{ k=0}^{\infty}\frac{i^{2k+1}(\tau m)^{k}}{(2k+1)!}\tilde{M}\nonumber\\
		&=& \sum_{ k=0}^{\infty}\frac{(-1)^{k}(\sqrt{\tau m})^{2k}}{(2k)!}I+ \frac{i}{\sqrt {\tau m}}\sum_{ k=0}^{\infty}\frac{(-1)^{k}(\sqrt {\tau m})^{2k+1}}{(2k+1)!}\tilde{M} \nonumber\\
		&=& \cos(\sqrt{\tau m})I+\frac{i}{\sqrt {\tau m}}\sin({\sqrt{\tau m}})\tilde{M}.
	\end{eqnarray}
	Simple calculations will give us 
	\begin{equation}\label{c12}
		\|\tilde{M}\|\leq T(1+\beta^{2}\|a\|_{L^{\infty}([0,T])}+\|b\|_{L^{\infty}([0,T])}).
	\end{equation}
	From \eqref{estref1} and \eqref{c12}, we have
	\begin{equation}\label{es1i}
		\|e^{i\int_{s}^{t}\left(\beta^{2} A(\tau)+B(\tau)\right)d\tau}\|\leq 1+T(1+\beta^{2}\|a\|_{L^{\infty}([0,T])}+\|b\|_{L^{\infty}([0,T])}).
	\end{equation}
	Now, when $m=0$, we have 
	$\tilde{M}=\left(\begin{matrix}
		0&\tau\\
		0&0
	\end{matrix}\right)$ and $e^{i\tilde{M}}=I+i\tau \left(\begin{matrix} 0&1\\
		0&0\end{matrix}\right),$
	so that in this case we have
	\begin{equation}\label{es1ii}
		\|e^{i\int_{s}^{t}\left(\beta^{2} A(\tau)+B(\tau)\right)d\tau}\|\leq 1+T.
	\end{equation}
	Thus from \eqref{es1i} and \eqref{es1ii}, we get
	\begin{equation}\label{estM}
		\|e^{i\int_{s}^{t}\left(\beta^{2} A(\tau)+B(\tau)\right)d\tau}\|\leq C(1+\beta^{2}),
	\end{equation}
	where $C=\max\{1+T+T\|b\|_{L^{\infty}([0,T])},T\|a\|_{L^{\infty}([0,T])}\}.$ Combining \eqref{c0} and \eqref{estM}, we get
	\begin{eqnarray}
		|W(t)|^2&\leq& 2|V_{I}(t)|^{2}+2C^{2}(1+\beta^{2})^{2}\|g\|^2_{L^2([0,T])}\nonumber\\
		&\lesssim& |V_{I}(t)|^{2}+(1+\beta^{2})^{2}\|g\|^2_{L^2([0,T])},
	\end{eqnarray}
	for all $t\in[0,T]$, with the constant independent of $\beta$. This completes the proof.
\end{proof}

\section{Proof of the main result}
\label{SEC:proof}

We now proceed to prove Theorem \ref{maintheo}. We first recall a few facts concerning the Fourier analysis on the lattice $\hbar\Zn$.
The Schwartz space $\mathcal{S}\left(\hbar\mathbb{Z}^{n}\right)$ on the lattice $\hbar\mathbb{Z}^{n}$ is the space of rapidly decreasing functions $u: \hbar\mathbb{Z}^{n} \rightarrow \mathbb{C}$, that is, $u \in \mathcal{S}\left(\hbar\mathbb{Z}^{n}\right)$ if for any $L<\infty$ there exists a constant $C_{u, L,\hbar}$ such
that
\begin{equation}
	|u(k)| \leq C_{u, L,\hbar}(1+|k|)^{-L}, \quad \text { for all } k \in \hbar\mathbb{Z}^{n} ,
\end{equation}
where $|k|=\hbar\left(\sum\limits_{j=1}^{n}k^2_j\right)^{1/2}$. The Fourier transform  $\widehat{u}$ of $u:\hbar\mathbb{Z}^{n}\to\mathbb C$ is defined as 
\begin{equation}\label{EQ:ftlat}
\widehat{u}(\theta)=\sum_{k\in \hbar\mathbb{Z}^{n}}u(k)e^{-2\pi\frac{ i}{\hbar}k\cdot \theta},~~~\theta\in\mathbb{T}^{n},~k\in \hbar\mathbb{Z}^{n}.
\end{equation} 
The Plancherel formula takes the form
\begin{eqnarray}\label{EQ:Planch}
\int_{\mathbb{T}^{n}}\left|\widehat{u}(\theta)\right|^{2}d\theta &=& \int_{\mathbb{T}^{n}}\widehat{u}(\theta)\overline{\widehat{u}(\theta)}d\theta\nonumber\\
&=&\int_{\mathbb{T}^{n}}\sum_{k\in \hbar \mathbb{Z}^{n}}u(k)e^{-\frac{2\pi i}{h}k\cdot\theta}\sum_{l\in \hbar \mathbb{Z}^{n}}\overline{u(l)}e^{\frac{2\pi i}{h}l\cdot\theta}d\theta\nonumber\\
&=& \sum_{k\in \hbar \mathbb{Z}^{n}}\left|u(k)\right|^{2}.
\end{eqnarray}
The scalar product in the Hilbert space $\l2h$ is given by
$$
(u,v)_{l^{2}(\hbar\mathbb{Z}^{n})}=\sum_{k\in \hbar\mathbb{Z}^{n}}u(k)\overline{v(k)}.
$$

We note that the symbol of $\mathcal{L}_{\hbar}$ defined by 
$\sigma_{\mathcal{L}_{\hbar}}(k,\theta)=e^{-2\pi\frac{ i}{\hbar}k\cdot \theta}
\mathcal{L}_{\hbar}(e^{2\pi\frac{ i}{\hbar}k\cdot \theta})$
is given by
\begin{equation}\label{EQ:Lhsymb}
\sigma_{\mathcal{L}_{\hbar}}(k,\theta)=\sum_{j=1}^{n}\left(e^{2\pi i \theta_j}+e^{-2\pi i \theta_j}\right)-2n=2\sum_{j=1}^{n}\cos(2\pi\theta_j)-2n,
\end{equation} 
with $(k,\theta)\in\hbar\mathbb Z^n\times\Tn$. 
We also note that it is independent of $k$ and $\hbar$.

\begin{proof}[Proof of Theorem \ref{maintheo}] 
Our aim is to reduce the Cauchy problem  \eqref{pde} to a form
allowing us to apply Corollary \ref{eqq11}. In order to do this, we take the Fourier
transform of \eqref{pde} with respect to $k\in{\hbar\mathbb{Z}^n}$, which gives
\begin{equation}\label{Feqn}
\partial^{2}_{t}\widehat{u}(t,\theta)-\hbar ^{-2}a(t)\sigma_{\mathcal{L}_{\hbar}}(k,\theta)\widehat{u}(t,\theta)+b(t)\widehat{u}(t,\theta)=\widehat{f}(t, \theta),\quad \theta\in\mathbb{T}^{n}.\end{equation}
Formally recalling the notation used in Proposition \ref{odelem} and Corollary \ref{eqq11}, we write
\begin{equation}\label{EQ:not1}
v(t):=\widehat{u}(t,\theta),\quad
\beta^{2}:= -\hbar^{-2}\sigma_{\mathcal{L}_{\hbar}}(k,\theta)=\hbar ^{-2}\left(2n-2\sum_{j=1}^{n}\cos(2\pi\theta_j)\right),
\end{equation} 
as well as
\begin{equation}\label{EQ:not2}
v_{0}:=\widehat{u}_{0}(\theta),\quad v_{1}:=\widehat{u}_{1}(\theta),\quad g(t):=\widehat f(t,\theta).\end{equation}
Therefore, equation \eqref{Feqn} becomes
$$v^{\prime\prime}(t)+\beta^{2}a(t)v(t)+b(t)v(t)=g(t),\quad t\in[0,T]$$
with $\beta=\beta_{\hbar,\theta}$ and all other functions depending on $\theta$ as a parameter.
We proceed by discussing implications of Corollary \ref{eqq11} separately in each case.

\medskip
\textbf{Case 1:}~${a\in \textbf{Lip}([0,T]),~ a(t)\geq a_0>0.}$

\smallskip
Applying Corollary \ref{eqq11} and inequality \eqref{EQ:ODE1} in Proposition \ref{odelem}, we get
\begin{equation}\label{cc1:eq1}
|v(t)|^{2}+|v^{\prime}(t)|^{2}\lesssim
  C_{1}e^{K_{1}T\langle\beta\rangle^{2}}( |v_0|^2+|v_1|^2)+(1+\beta^2)^{2}\|g\|^2_{L^2([0,T])}.
\end{equation} 

 Recalling the notation \eqref{EQ:not1} and \eqref{EQ:not2}, in \eqref{cc1:eq1}, we get 
\begin{multline}\label{mt:c1}
	|\widehat{u}(t,\theta)|^{2}+|\partial_{t}\widehat{u}(t,\theta)|^{2}\lesssim  C_{1}e^{K_{1}T\left(1-\hbar^{-2}\sigma_{\mathcal{L}_{\hbar}}(k,\theta)\right)} (|\widehat{u}_0(\theta)|^{2}+|\widehat{u}_1(\theta)|^{2})+\\
	(1-\hbar^{-2}\sigma_{\mathcal{L}_{\hbar}}(k,\theta))^{2}\|\widehat{f}(\cdot,\theta)\|^2_{L^2([0,T])}.
\end{multline}
Now, by using the Plancherel formula \eqref{EQ:Planch} and Fourier transform, we get
\begin{multline}
\|u(t)\|^{2}_{\l2h}+\|\partial_{t}{u}(t)\|^{2}_{\l2h}
\lesssim  C_{1}\left(\|e^{B_{1}T\left(I-\hbar^{-2}\mathcal{L}_{\hbar}\right)}u_0\|^{2}_{\l2h}\right.+\\
\left.\|e^{B_{1}T\left(I-\hbar^{-2}\mathcal{L}_{\hbar}\right)}{u}_1\|^{2}_{\l2h}\right)+
\|(I-\hbar^{-2}\mathcal{L}_{\hbar})f\|^2_{L^2([0,T];\l2h)},
\end{multline}
where $B_{1}=\frac{K_{1}}{2}>0.$
We also note that the symbol  $\sigma_{\mathcal{L}_{\hbar}}$ is of order zero, therefore  $I-\hbar^{-2}\mathcal{L}_{\hbar}$ and $e^{B_{1}T\left(I-\hbar^{-2}\mathcal{L}_{\hbar}\right)}$ are bounded pseudo-difference  operators on $\ell^{2}(\hbar\mathbb{Z}^{n})$,  associated with the symbols $1-\hbar^{-2}\sigma_{\mathcal{L}_{\hbar}}$ and $e^{B_{1}T\left(1-\hbar^{-2}\sigma_{\mathcal{L}_{\hbar}}\right)}$, respectively. Therefore
\begin{eqnarray}
	\|u(t)\|^{2}_{\l2h}+\|\partial_{t}{u}(t)\|^{2}_{\l2h}
	&\lesssim&  C_{1}\|e^{B_{1}T\left(I-\hbar^{-2}\mathcal{L}_{\hbar}\right)}\|^{2}\left(\|u_0\|^{2}_{\l2h}+\|u_{1}\|^{2}_{\l2h}\right)+\nonumber\\
	&&\|I-\hbar^{-2}\mathcal{L}_{\hbar}\|^{2}\|f\|^2_{L^2([0,T];\l2h)}\nonumber\\
	&\leq&C_{\hbar,T}\left(\|u_0\|^{2}_{\l2h}+\|u_{1}\|^{2}_{\l2h}+\|f\|^2_{L^2([0,T];\l2h)}\right),\nonumber\\
\end{eqnarray}
where $C_{\hbar,T}=\max\{C_{1}\|e^{B_{1}T\left(I-\hbar^{-2}\mathcal{L}_{\hbar}\right)}\|^{2}_{\mathcal{L}(\ell^{2}(\hbar\mathbb{Z}^{n}))},\|I-\hbar^{-2}\mathcal{L}_{\hbar}\|^{2}_{\mathcal{L}(\ell^{2}(\hbar\mathbb{Z}^{n}))}\}$.\\

\medskip
\textbf{Case 2:} ${a\in \mathcal{C}^{\alpha}([0,T])}$, with ${0<\alpha<1,\; a_0:=\min_{[0,T]} a(t)>0.}$
\smallskip

The application of Corollary \ref{eqq11} and inequality \eqref{odecase2} in Proposition \ref{odelem}, gives
\begin{equation}\label{cc2:eq1}
|v(t)|^{2}+|v^{\prime}(t)|^{2}\lesssim
C_{2}e^{K_{2}T\langle\beta\rangle^{\frac{1}{s}}}( |v_0|^2+|v_1|^2)+(1+\beta^2)^{2}\|g\|^2_{L^2([0,T])},
\end{equation}
for $0<s\leq \frac{1}{2}$.
Using \eqref{EQ:not1} and \eqref{EQ:not2} in \eqref{cc2:eq1}, and similar to previous case, we get
\begin{eqnarray}
\|u(t)\|^{2}_{\l2h}+\|\partial_{t}{u}(t)\|^{2}_{\l2h}
&\lesssim&  C_{2}\|e^{B_{2}T\left(I-\hbar^{-2}\mathcal{L}_{\hbar}\right)^{\frac{1}{2s}}}\|^{2}\left(\|u_0\|^{2}_{\l2h}+\|u_{1}\|^{2}_{\l2h}\right)\nonumber\\
&&+\|I-\hbar^{-2}\mathcal{L}_{\hbar}\|^{2}\|f\|^2_{L^2([0,T];\l2h)}\nonumber\\
&\leq& C_{\hbar,T}\left(\|u_0\|^{2}_{\l2h}+\|u_{1}\|^{2}_{\l2h}+\|f\|^2_{L^2([0,T];\l2h)}\right),\nonumber\\
\end{eqnarray}
where $C_{\hbar,T}=\max\{C_{2}\|e^{B_{2}T\left(I-\hbar^{-2}\mathcal{L}_{\hbar}\right)^{\frac{1}{2s}}}\|^{2}_{\mathcal{L}(\ell^{2}(\hbar\mathbb{Z}^{n}))},\|I-\hbar^{-2}\mathcal{L}_{\hbar}\|^{2}_{\mathcal{L}(\ell^{2}(\hbar\mathbb{Z}^{n}))}\}$.\\

\medskip
\textbf{Case 3:} $a\in\mathcal{C}^{l}([0,T]),$ with $l\geq 2,$ $a(t)\geq 0.$
\smallskip

Similar to the above, from Corollary \ref{eqq11} and inequality \eqref{odecase3} in Proposition \ref{odelem},   we get
\begin{equation}\label{cc3:eq0}
|v(t)|^{2}+|v'(t)|^{2}\lesssim C_{3} e^{K_{3}\langle\beta\rangle^{6-\frac{4}{\sigma}}}(|v_{0}|^{2}+|v_{1}|^{2})+
C(1+\beta^2)^{2}\|g\|^2_{L^2([0,T])}
\end{equation} 
for $\sigma=1+\frac{l}{2}.$
At the same time, recalling the notation \eqref{EQ:not1} and \eqref{EQ:not2}, in \eqref{cc3:eq0}, and similar to Case 1, we get 
\begin{eqnarray}
\|u(t)\|^{2}_{\l2h}+\|\partial_{t}{u}(t)\|^{2}_{\l2h}
&\lesssim&  C_{3}^{'}\|e^{B_{3}\left(I-\hbar^{-2}\mathcal{L}_{\hbar}\right)^{3-\frac{2}{\sigma}}}\|^{2}\left(\|u_0\|^{2}_{\l2h}+\|u_{1}\|^{2}_{\l2h}\right)+\nonumber\\
&&\|I-\hbar^{-2}\mathcal{L}_{\hbar}\|^{2}\|f\|^2_{L^2([0,T];\l2h)}\nonumber\\
&\leq& C_{\hbar,T}\left(\|u_0\|^{2}_{\l2h}+\|u_{1}\|^{2}_{\l2h}+\|f\|^2_{L^2([0,T];\l2h)}\right),\nonumber\\
\end{eqnarray}
where $C_{\hbar,T}=\max\{C_{3}^{'}\|e^{B_{3}\left(I-\hbar^{-2}\mathcal{L}_{\hbar}\right)^{3-\frac{2}{\sigma}}}\|^{2}_{\mathcal{L}(\ell^{2}(\hbar\mathbb{Z}^{n}))},\|I-\hbar^{-2}\mathcal{L}_{\hbar}\|^{2}_{\mathcal{L}(\ell^{2}(\hbar\mathbb{Z}^{n}))}\}$.\\

\medskip
\textbf{Case 4:} $a\in \mathcal{C}^{\alpha}([0,T])$, with $0<\alpha<2$, $a(t)\geq 0.$
\smallskip

Similar to the previous cases from Corollary \ref{eqq11} and inequality \eqref{odecase4} in Proposition \ref{odelem}, we get
\begin{equation}\label{cc3:eq1}
|v(t)|^{2}+|v^{\prime}(t)|^{2}\lesssim
C_{4}e^{K_{4}T\langle\beta\rangle^{\frac{1}{s}}}( |v_0|^2+|v_1|^2)+C(1+\beta^2)^{2}\|g\|^2_{L^2([0,T])},
\end{equation} 
for $0<s\leq \frac{1}{2}+\frac{\alpha}{2}.$
At the same time, recalling the notation \eqref{EQ:not1} and \eqref{EQ:not2} in \eqref{cc3:eq1}, then similar to the Case 1, we get
\begin{eqnarray}
\|u(t)\|^{2}_{\l2h}+\|\partial_{t}{u}(t)\|^{2}_{\l2h}
&\lesssim&  C_{4}\|e^{B_{4}T\left(I-\hbar^{-2}\mathcal{L}_{\hbar}\right)^{\frac{1}{2s}}}\|^{2}\left(\|u_0\|^{2}_{\l2h}+\|u_{1}\|^{2}_{\l2h}\right)+\nonumber\\
&&\|I-\hbar^{-2}\mathcal{L}_{\hbar}\|^{2}\|f\|^2_{L^2([0,T];\l2h)}\nonumber\\
&\leq& C_{\hbar,T}\left(\|u_0\|^{2}_{\l2h}+\|u_{1}\|^{2}_{\l2h}+\|f\|^2_{L^2([0,T];\l2h)}\right),\nonumber\\
\end{eqnarray}
where $C_{\hbar,T}=\max\{C_{4}\|e^{B_{4}T\left(I-\hbar^{-2}\mathcal{L}_{\hbar}\right)^{\frac{1}{2s}}}\|^{2}_{\mathcal{L}(\ell^{2}(\hbar\mathbb{Z}^{n}))},\|I-\hbar^{-2}\mathcal{L}_{\hbar}\|^{2}_{\mathcal{L}(\ell^{2}(\hbar\mathbb{Z}^{n}))}\}$.  This proves \eqref{EQ:est0} in Theorem \ref{maintheo}. We will consider \eqref{EQ:conv} in the next section.
\end{proof}

\section{Limit as $\hbar\to 0$}
\label{SEC:limit}

In this section, we compare the solutions of \eqref{pde} on $\hbar\Zn$ as $\hbar\to 0$, with the solutions of \eqref{EQ:wern} on $\Rn$. We now proceed to prove Theorem \ref{limthm}.
\begin{proof}[Proof of Theorem \ref{limthm}]
Consider two Cauchy problems:
\begin{equation}\label{CP1}
	\left\{
	\begin{array}{ll}
		\partial^{2}_{t}u(t,k)-a(t)\hbar^{-2}\mathcal{L}_{\hbar}u(t,k)+b(t)u(t,k)=g(t,k), \quad k\in \hbar\mathbb{Z}^n,\\
		u(0,k)=u_{0}(k),\\
		\partial_{t}u(0,k)=u_1(k),
	\end{array}
	\right.
\end{equation}
and
\begin{equation}\label{CP2}
	\left\{
	\begin{array}{ll}
		\partial^{2}_{t}v(t,x)-a(t)\mathcal{L}v(t,x)+b(t)v(t,x)=g(t,x), \quad x\in\mathbb{R}^n,\\
		v(0,x)=u_{0}(x),\\
		\partial_{t}v(0,x)=u_1(x),
	\end{array}
	\right.
\end{equation}
where  $\mathcal{L}$ is the Laplacian on $\Rn$.
From the equations \eqref{CP1} and \eqref{CP2}, 
denoting $w:=u-v$,  we get
\begin{equation}\label{CPF}
	\left\{
	\begin{array}{ll}
		\partial^{2}_{t}w(t,k)-a(t)\hbar^{-2}\mathcal{L}_{\hbar}w(t,k)+b(t)w(t,k)=a(t)\left(\hbar^{-2}\mathcal{L}_{\hbar}-\mathcal{L}\right)v(t,k),\quad k\in\hbar\mathbb{Z}^{n},\\
		w(0,k)=0,\\
		\partial_{t}w(0,k)=0.               
	\end{array}
	\right.
\end{equation}
Our aim is to reduce the Cauchy problem  \eqref{CPF} to a form
allowing us to apply Corollary \ref{eqq11}. In order to do this, we take the Fourier
transform of \eqref{CPF} with respect to $k\in{\hbar\mathbb{Z}^n}$. This gives
\begin{equation}\label{weq}
	\partial^{2}_{t}\widehat{w}(t,\theta)-\hbar ^{-2}a(t)\sigma_{\mathcal{L}_{\hbar}}(k,\theta)\widehat{w}(t,\theta)+b(t)\widehat{w}(t,\theta)=\widehat{f}(t, \theta),\quad \theta\in\mathbb{T}^{n},\end{equation}
where $f(t,k):=a(t)\left(\hbar^{-2}\mathcal{L}_{\hbar}-\mathcal{L}\right)v(t,k)$, and $\sigma_{\mathcal{L}_{\hbar}}(k,\theta)=\sigma_{\mathcal{L}_{\hbar}}(\theta)$ is independent of $k\in\hbar\mathbb{Z}^{n}$. Since $w_{0}=w_{1}=0$, from Proposition \ref{odelem}, it follows that the solution of corresponding homogeneous equation of \eqref{CPF} is identically zero. Now applying Corollary \ref{eqq11} in \eqref{weq}, we get 
\begin{equation}
	|\widehat{w}(t,\theta)|^{2}+|\partial_{t}\widehat{w}(t,\theta)|^{2}\lesssim  (1-\hbar^{-2}\sigma_{\mathcal{L}_{\hbar}}(k,\theta))^{2}\|\widehat{f}(\cdot,\theta)\|^2_{L^2([0,T])}.
\end{equation}
Similar to the proof of \eqref{EQ:est0}, by recalling  Plancherel formula \eqref{EQ:Planch} and Fourier transform, we have
\begin{multline}\label{WP11}
	\left\|w(t)\right\|^{2}_{\l2h} +\left\|\partial_t w(t)\right\|^{2}_{\l2h}
	\leq C\left\|\left(I-\hbar^{-2}\mathcal{L}_{\hbar}\right)a(\hbar^{-2}\mathcal{L}_{\hbar}-\mathcal{L})v\right\|^{2}_{L^{2}([0,T];\ell^{2}(\hbar\mathbb{Z}^{n}))}\\
	\leq C\left\|a\right\|^{2}_{L^{\infty}([0,T])}\sup\limits_{t\in[0,T]}\left\|\left(I-\hbar^{-2}\mathcal{L}_{\hbar}\right)(\hbar^{-2}\mathcal{L}_{\hbar}-\mathcal{L})v(t,\cdot)\right\|^{2}_{\ell^{2}(\hbar\mathbb{Z}^{n})},
\end{multline}
where $C>0$ is independent of $\hbar$.

Now we will estimate the term $\left\|\left(I-\hbar^{-2}\mathcal{L}_{\hbar}\right)(\hbar^{-2}\mathcal{L}_{\hbar}-\mathcal{L})v(t,\cdot)\right\|^{2}_{\ell^{2}(\hbar\mathbb{Z}^{n})}$.
Let $\phi\in C^{4}(\mathbb{R}^{n})$, then by Taylor's theorem with the Lagrange's form of the remainder, we have
\begin{equation}\label{tylr}
	\phi(\xi+\mathbf{h})=\sum_{|\alpha|\leq 3} \frac{\partial^{\alpha} \phi(\xi)}{\alpha !} \mathbf{h}^{\alpha}+\sum_{|\alpha|=4} \frac{\partial^{\alpha} \phi(\xi+\theta_{\xi}\mathbf{h})}{\alpha !} \mathbf{h}^{\alpha},
\end{equation} 
for some $\theta_{\xi}\in (0,1)$ depending on $\xi$.  Let $v_j$ be the $j^{th}$ basis vector in $\Zn$, having all zeros except for $1$ as the $j^{th}$ component and then by taking  $\mathbf{h}=v_{j}$ and $-v_{j}$ in \eqref{tylr}, we have
\begin{eqnarray}\label{d1}
	\phi(\xi+v_{j})=\phi(\xi)+\phi^{(v_{j})}(\xi)+\frac{1}{2!}\phi^{(2v_{j})}(\xi)+\frac{1}{3!}\phi^{(3v_{j})}(\xi)+\frac{1}{4!}\phi^{(4v_{j})}(\xi+\theta_{j,\xi} v_{j}),
\end{eqnarray} and
\begin{eqnarray}\label{d2}
	\phi(\xi-v_{j})=\phi(\xi)-\phi^{(v_{j})}(\xi)+\frac{1}{2!}\phi^{(2v_{j})}(\xi)-\frac{1}{3!}\phi^{(3v_{j})}(\xi)+\frac{1}{4!}\phi^{(4v_{j})}(\xi-\tilde{\theta}_{j,\xi} v_{j}),\end{eqnarray}
for some $\theta_{j,\xi},\tilde{\theta}_{j,\xi}\in(0,1)$. Using \eqref{d1} and \eqref{d2}, we have
\begin{equation*}
	\phi(\xi+v_{j})+\phi(\xi-v_{j})-2\phi(\xi)=\phi^{(2v_{j})}(\xi)+\frac{1}{4!}\left(\phi^{(4v_{j})}(\xi+\theta_{j,\xi} v_{j})+\phi^{(4v_{j})}(\xi-\tilde{\theta}_{j,\xi} v_{j})\right).
\end{equation*}
Since $\delta_{\xi_{j}}^{2}\phi(\xi)=\phi(\xi+v_{j})+\phi(\xi-v_{j})-2\phi(\xi)$, where $\delta_{\xi_{j}}\phi(\xi):=\phi(\xi+\frac{1}{2}v_{j})-\phi(\xi-\frac{1}{2}v_{j}),$ is the usual central difference operator, it follows that
\begin{eqnarray}\label{d3}
	\delta_{\xi_{j}}^{2}\phi(\xi)&=&\phi^{(2v_{j})}(\xi)+\frac{1}{4!}\left(\phi^{(4v_{j})}(\xi+\theta_{j,\xi} v_{j})+\phi^{(4v_{j})}(\xi-\tilde{\theta}_{j,\xi}v_{j})\right).
\end{eqnarray}
Now by adding all the above $n$-equations for $j=1,\dots,n$, we get
\begin{eqnarray}\label{d3}
	\sum\limits_{j=1}^{n}\delta_{\xi_{j}}^{2}\phi(\xi)=\sum\limits_{j=1}^{n}\phi^{(2v_{j})}(\xi)+\frac{1}{4!}\sum\limits_{j=1}^{n}\left(\phi^{(4v_{j})}(\xi+\theta_{j,\xi} v_{j})+\phi^{(4v_{j})}(\xi-\tilde{\theta}_{j,\xi} v_{j})\right).
\end{eqnarray}
Let us define a function $V_{\theta_{j}v_{j}}\phi:\mathbb{R}^{n}\to \mathbb{R}$  by $V_{\theta_{j}v_{j}}\phi(\xi):=\phi(\xi-\theta_{j,\xi}v_{j}),$ then we get
\begin{eqnarray}
	\sum\limits_{j=1}^{n}\delta_{\xi_{j}}^{2}\phi(\xi)-\sum\limits_{j=1}^{n}\frac{\partial^{2}}{\partial\xi_{j}^{2}}\phi(\xi)=\frac{1}{4!}\sum\limits_{j=1}^{n}\left(V_{-\theta_{j}v_{j}}\phi^{(4v_{j})}(\xi)+V_{\tilde{\theta}_{j}v_{j}}\phi^{(4v_{j})}(\xi)\right).
\end{eqnarray}
Now we extend this to $\hbar\Zn$. Consider a function  $\phi_{\hbar}:\mathbb{R}^{n}\to \mathbb{R}$  defined by
$\phi_{\hbar}(\xi):=\phi(\hbar\xi)$. Clearly  $\phi_{\hbar}\in C^{4}(\mathbb{R}^{n})$, if we take $\phi\in C^{4}(\mathbb{R}^{n})$. Now we have
\begin{eqnarray}\label{eqhzn}
	\mathcal{L}_{1}\phi_{\hbar}(\xi)-\mathcal{L}\phi_{\hbar}(\xi)=\frac{1}{4!}\sum\limits_{j=1}^{n}\left(V_{-\theta_{j}v_{j}}\phi_{\hbar}^{(4v_{j})}(\xi)+V_{\tilde{\theta}_{j}v_{j}}\phi_{\hbar}^{(4v_{j})}(\xi)\right),
\end{eqnarray}
where $\mathcal{L}$ is the Laplacian on $\Rn$ and $\mathcal{L}_{1}$ is the discrete difference Laplacian on $\Zn$.  One can quickly notice that 
\begin{equation}
	V_{-\theta_{j}v_{j}}\phi_{\hbar}^{(4v_{j})}(\xi)=\phi_{\hbar}^{(4v_{j})}(\xi+\theta_{j,\xi} v_{j})=\hbar^{4}\phi^{(4v_{j})}(\hbar\xi+\hbar\theta_{j,\xi} v_{j})=\hbar^{4}V_{-\hbar\theta_{j}v_{j}}\phi^{(4 v_{j})}(\hbar\xi).
\end{equation}
Therefore, the equation \eqref{eqhzn} becomes
\begin{eqnarray}
	\left(\mathcal{L}_{\hbar}-\hbar^{2}\mathcal{L}\right)\phi(\hbar\xi)=\frac{\hbar^{4}}{4!}\sum\limits_{j=1}^{n}\left(V_{-\hbar\theta_{j}v_{j}}\phi^{(4v_{j})}(\hbar\xi)+V_{\hbar\tilde{\theta}_{j}v_{j}}\phi^{(4v_{j})}(\hbar\xi)\right),
\end{eqnarray}
whence we get
\begin{multline}
	\left(I-\hbar^{-2}\mathcal{L}_{\hbar}\right)\left(\mathcal{L}_{\hbar}-\hbar^{2}\mathcal{L}\right)\phi(\hbar\xi)=\frac{\hbar^{4}}{4!}\sum\limits_{j=1}^{n}\left(\left(I-\hbar^{-2}\mathcal{L}_{\hbar}\right)V_{-\hbar\theta_{j}v_{j}}\phi^{(4v_{j})}(\hbar\xi)\right.+\\
	\left.\left(I-\hbar^{-2}\mathcal{L}_{\hbar}\right)V_{\hbar\tilde{\theta}_{j}v_{j}}\phi^{(4v_{j})}(\hbar\xi)\right).
\end{multline}
Hence, it follows that
\begin{multline}\label{EQ:convh}
	\left\|\left(I-\hbar^{-2}\mathcal{L}_{\hbar}\right)(\hbar^{-2}\mathcal{L}_{\hbar}-\mathcal{L})\phi\right\|^{2}_{\ell^{2}(\hbar\mathbb{Z}^{n})}\lesssim\hbar^{4}\left[\max_{1\leq j\leq n}\left\|\left(I-\hbar^{-2}\mathcal{L}_{\hbar}\right)V_{-\hbar\theta_{j}v_{j}}\phi^{(4v_{j})}\right\|^{2}_{\ell^{2}(\hbar\mathbb{Z}^{n})}\right.\\
	+\left.\max_{1\leq j\leq n}\left\|\left(I-\hbar^{-2}\mathcal{L}_{\hbar}\right)V_{\hbar\tilde{\theta}_{j}v_{j}}\phi^{(4v_{j})}\right\|^{2}_{\ell^{2}(\hbar\mathbb{Z}^{n})}\right].
\end{multline}
Since in Case 1, $v(t,\cdot)\in H^{s+1}(\Rn)\cap H^{s}(\Rn)= H^{s+1}(\Rn)$, by Sobolev embedding theorem $\left(\right.$see e.g.\cite[Excercise 2.6.17]{Ruzhansky-Turunen:BOOK}$\left.\right)$, we have 
\begin{equation}
	s>k+\frac{n}{2}\implies H^{s}(\Rn)\subseteq C^{k}(\Rn).
\end{equation}
Thus for $s>3+\dfrac{n}{2},$ we have $H^{s+1}(\Rn)\subseteq C^{4}(\Rn)$, and if $v(t,\cdot)\in H^{s+1}(\Rn)$, then $v^{(4v_{j})}(t,\cdot)\in H^{s-3}(\Rn) \subseteq H^{2}(\Rn)$ whenever   $s\geq 5$,  which gives
\begin{equation}\label{norm}
	\lim\limits_{\hbar \to 0}\left\|\left(I-\hbar^{-2}\mathcal{L}_{\hbar}\right)V_{-\hbar\theta_{j}v_{j}}v^{(4v_{j})}(t,\cdot)\right\|^{2}_{\ell^{2}(\hbar\mathbb{Z}^{n})}=	\left\|(I-\mathcal{L})v^{(4v_{j})}(t,\cdot)\right\|^{2}_{L^{2}(\mathbb{R}^{n})}<\infty.
\end{equation}
Combining the above assumptions, we can choose $s\geq 5$ for $n\leq3$ and $s>3+\frac{n}{2}$ for $n\geq 4$. On the other hand for Case 2, 3 and 4, we have $v(t,\cdot)\in \gamma^{s}(\mathbb{R}^{n})\cap H^{p+1}(\mathbb{R}^{n})\subseteq C^{\infty}(\mathbb{R}^{n})$, under the assumptions \eqref{EQ:rn1}-\eqref{EQ:rn3} and $p\geq 5$. Now from  \eqref{EQ:convh}, it follows that
\begin{multline}\label{EQ:cht}
	\left\|\left(I-\hbar^{-2}\mathcal{L}_{\hbar}\right)(\mathcal{L}-\hbar^{-2}\mathcal{L}_{\hbar})v(t,\cdot)\right\|^{2}_{\l2h}\lesssim\\
	\hbar^{4}\left[\max_{1\leq j\leq n}\left\|\left(I-\hbar^{-2}\mathcal{L}_{\hbar}\right)V_{-\hbar\theta _{j}v_{j}}v^{(4v_{j})}(t,\cdot)\right\|^{2}_{\l2h}\right.+\\
	\left.\max_{1\leq j\leq n}\left\|\left(I-\hbar^{-2}\mathcal{L}_{\hbar}\right)V_{\hbar\tilde{\theta}_{j}v_{j}}v^{(4v_{j})}(t,\cdot)\right\|^{2}_{\l2h}\right].
\end{multline}
Using \eqref{WP11}, \eqref{norm} and \eqref{EQ:cht}, we get $\left\|w(t)\right\|^{2}_{\l2h}+\left\|\partial_t w(t)\right\|^{2}_{\l2h}\to 0$ as $\hbar\to 0.$ Hence $\left\|w(t)\right\|_{\l2h} \to 0$ and $\left\|\partial_t w(t)\right\|_{\l2h} \to 0$ as $\hbar \to 0$. This finishes the proof of Theorem \ref{limthm}.	
\end{proof}

\bibliographystyle{abbrv}

\end{document}